\documentclass{amsart}

\usepackage{amsmath,amssymb,amsthm}

\usepackage{pinlabel}

\newtheorem{prop}{Proposition}[section]
\newtheorem{thm}[prop]{Theorem}
\newtheorem{lem}[prop]{Lemma}

\theoremstyle{definition}

\newtheorem{defn}[prop]{Definition}
\newtheorem{ex}[prop]{Example}

\newtheorem{rem}[prop]{Remark}

\newtheorem*{ack}{Acknowledgements}

\def\co{\colon\thinspace}

\newcommand{\C}{\mathbb{C}}
\newcommand{\CP}{\mathbb{C}\mathrm{P}}

\newcommand{\rmd}{\mathrm{d}}

\newcommand{\rme}{\mathrm{e}}

\newcommand{\bbH}{\mathbb{H}}

\newcommand{\rmi}{\mathrm{i}}

\newcommand{\rmj}{\mathrm{j}}

\newcommand{\rmk}{\mathrm{k}}
\newcommand{\calK}{\mathcal{K}}

\newcommand{\omu}{\overline{\mu}}

\newcommand{\N}{\mathbb{N}}

\newcommand{\PD}{\mathrm{PD}}

\newcommand{\R}{\mathbb{R}}
\newcommand{\RP}{\mathbb{R}\mathrm{P}}

\newcommand{\SL}{\mathrm{SL}}

\newcommand{\ttt}{\mathtt{t}}

\newcommand{\oz}{\overline{z}}
\newcommand{\Z}{\mathbb{Z}}

\newcommand{\xist}{\xi_{\mathrm{st}}}

\DeclareMathOperator{\Int}{Int}
\DeclareMathOperator{\Crit}{Crit}
\DeclareMathOperator{\sign}{sign}

\begin{document}

\author[H.~Geiges]{Hansj\"org Geiges}
\address{Mathematisches Institut, Universit\"at zu K\"oln,
Weyertal 86--90, 50931 K\"oln, Germany}
\email{geiges@math.uni-koeln.de}
\author[J.~Hedicke]{Jakob Hedicke}
\address{Centre de Recherches Math\'ematiques, Universit\'e de Montr\'eal,
Pavillon Andr\'e-Aisenstadt, 2920 Chemin de la tour, Montr\'eal H3T 1J4,
Canada}
\email{jakob.hedicke@gmail.com}
\author[M.~Sa\u{g}lam]{Murat Sa\u{g}lam}
\address{Mathematisches Institut, Universit\"at zu K\"oln,
Weyertal 86--90, 50931 K\"oln, Germany}
\email{msaglam@math.uni-koeln.de}

\title[Bott-integrable Reeb flows]{Bott-integrable Reeb flows on $3$-manifolds}

\date{}

\begin{abstract}
This paper is devoted to studying a notion of Bott integrability for Reeb
flows on contact $3$-manifolds. We show, in analogy with work of
Fomenko--Zieschang on Hamiltonian flows in dimension~$4$, that Bott-integrable
Reeb flows exist precisely on graph manifolds. We also show that
all $S^1$-invariant contact structures on Seifert manifolds, as well
as all contact structures on the $3$-sphere, on the $3$-torus, and
on $S^1\times S^2$, admit
Bott-integrable Reeb flows. Along the way, we establish some
general Liouville-type theorems for Bott-integrable Reeb flows,
and a number of topological constructions (connected sum,
open books, Dehn surgery) that may be expected to have wider
applications.
\end{abstract}



\maketitle


\section{Introduction}
It has been said that `a precise definition of the classical concept
of a completely integrable system is often elusive.'~\cite{webs03}.
This dictum applies to the symplectic setting, but it may be even more
apposite in the context of various notions of `contact integrability'
that have been proposed and studied. We provide an overview of
the literature at the end of this introduction.
\subsection{Bott integrability}
In dimension~$3$, fortunately, a favourite definition
of contact integrability suggests itself.
Let $\alpha$ be a positive contact form on a closed, oriented
$3$-manifold, that is, $\alpha\wedge\rmd\alpha>0$. We write
$R=R_{\alpha}$ for its Reeb vector field, defined by
$i_R\rmd\alpha=0$ and $\alpha(R)=1$. The
$\R$-invariant extension of $R$ to $\R\times M$ coincides with
the Hamiltonian vector field $X_H$ of the Hamiltonian
function $H\co \R\times M\rightarrow\R$, $(t,p)\mapsto \rme^t$
on the symplectisation $\bigl(\R\times M,\omega:=\rmd(\rme^t\alpha)\bigr)$.

It is therefore natural to define integrability (in the sense
of Arnold--Liouville) of the Reeb flow on $M$ in terms of the integrability
of the flow of $X_H$ on $\R\times M$. In the present paper we
restrict attention to the situation where the integrals of the
Hamiltonian flow are Morse--Bott functions. This leads,
as we shall see in Section~\ref{section:bott-int},
to the following definition.

\begin{defn}
\label{defn:reeb-bott-int}
The Reeb flow of $(M,\alpha)$ is called \textbf{Bott integrable}
if there is a Morse--Bott function $f\co M\rightarrow\R$
invariant under the Reeb flow, that is, $\rmd f(R)=0$.
The function $f$ is called a \textbf{Bott integral} of~$R$.
\end{defn}

Recall that a Morse--Bott function is a smooth function $f$ whose set
$\Crit(f):=\{p\in M\co\rmd_pf=0\}$ of critical points is a
submanifold (with components of various positive codimensions), and such
that the Hessian of $f$ is non-degenerate
in transverse directions along $\Crit(f)$.

\begin{rem}
We follow the well-established hyphenation rules as
in well-defined vs.\ well defined, i.e.\ we say `the flow
is Bott integrable', but speak of `a Bott-integrable flow'.
\end{rem}
\subsection{The main results}
Our first question about Bott-integrable Reeb flows is of topological nature:
which $3$-manifolds admit such a flow? (Without the
integrability condition, every $3$-manifold can arise.)
The answer is provided by the
following theorem. The notion of graph manifolds is due to
Waldhausen~\cite{wald67I,wald67II};
this is the class of $3$-manifolds that can be cut along tori
into Seifert fibred pieces. See Section~\ref{subsection:graph-manifold}
for more details.

\begin{thm}
\label{thm:bott-graph}
A closed, oriented $3$-manifold admits a Bott-integrable Reeb flow
if and only if it is a graph manifold.
\end{thm}

For Hamiltonian flows in dimension~$4$, this characterisation of the closed,
orientable $3$-manifolds that can arise as energy hypersurfaces of
a Bott-integrable Hamiltonian flow has been established by
Fomenko and Zieschang~\cite{fozi87}. For a comprehensive account
see the monograph by Bolsinov and Fomenko~\cite{bofo04}.
An alternative characterisation
in terms of manifolds admitting a decomposition along tori
into only two simple types of building blocks (see
Proposition~\ref{prop:graph-AB}) was given by Brailov and
Fomenko~\cite{brfo89}. \emph{A fortiori}, the `only if'
part of the theorem follows.
Our task, then, will be to establish the `if' statement;
this will be completed in Section~\ref{subsection:thm1}.

The next step is to ask if the subclass of integrable Reeb flows
is `large' within the class of integrable Hamiltonian
flows. Here it is reasonable to talk about Hamiltonian
flows on an energy hypersurface up to \emph{Liouville equivalence},
i.e.\ up to a diffeomorphism that preserves the (singular)
\emph{Liouville foliation} made up of the level sets of the Bott integral.
We restrict our attention to topologically stable integrable systems,
i.e.\ flows whose Liouville foliation does not change up to
diffeomorphism under small perturbations of the system.
In particular, such systems only have isolated periodic orbits as critical
levels of the Bott integral. For these topologically stable
flows the classification up to Liouville equivalence is described
in~\cite{bofo04}. The next theorem, which will be proved
in Section~\ref{subsection:liouville-eq}, says that each equivalence class
actually contains an integrable Reeb flow.

\begin{thm}
\label{thm:liouville-eq}
Any topologically stable flow on a closed $3$-dimensional energy hypersurface
of a $4$-dimensional Hamiltonian system admitting a Bott integral
is Liouville equivalent to a Bott-integrable Reeb flow.
\end{thm} 

\begin{rem}
For Euler flows and Reeb flows of stable Hamiltonian structures,
results analogous to our Theorems \ref{thm:bott-graph}
and~\ref{thm:liouville-eq} have been established by Cardona~\cite{card22}.
\end{rem}

Another sense in which integrable Reeb flows might or might not
be abundant concerns the class of contact structures that can
be realised. It is understood throughout
that we are dealing with \emph{positive}
contact structures on \emph{oriented} $3$-manifolds; observe that
in dimension $3$ the sign of the volume form $\alpha\wedge\rmd\alpha$
of any contact form $\alpha$ defining a given contact structure
$\xi=\ker\alpha$ is independent of the choice of~$\alpha$.

\begin{defn}
A positive contact structure $\xi$ on a closed, oriented $3$-manifold
is said to \textbf{admit a Bott-integrable Reeb flow} if there is a pair
$(\alpha,f)$ consisting of a contact form $\alpha$ defining
$\xi=\ker\alpha$, and a Bott integral $f$ of $R_{\alpha}$.
\end{defn}

Here we meet our first non-existence statement. By the work
of Macarini and Schlenk~\cite{masc11}, the canonical contact structure
on the unit cotangent bundle of a closed, oriented surface of genus at
least~$2$ admits Reeb flows of positive topological entropy only.
On the other hand, Paternain~\cite{pate91} has shown that if
the Hamiltonian flow on a $3$-dimensional energy hypersurface
of a $4$-dimensional integrable system has the property that
the critical levels of the integral constitute submanifolds --- which is
certainly satisfied in the Bott-integrable case ---, then the topological
entropy of this $3$-dimensional flow is zero. In fact, this connection with
systems of zero entropy is one of the motivations for looking at
Bott-integrable Reeb flows. By combining \cite{masc11} and \cite{pate91},
one obtains the following result.

\begin{prop}
\label{prop:no-bott}
The canonical contact structure on the unit cotangent bundle of a
closed, oriented surface of
genus at least~$2$ does not admit a Bott-integrable Reeb flow.\qed
\end{prop}

We collect other non-existence statements of this kind in
Section~\ref{subsection:non-exist}.

By contrast, we can prove several results showing that
Bott-integrable Reeb flows are far from scarce. Here is a simple
statement concerning a whole class of contact structures, which
will be proved in Section~\ref{section:S1invariant}. For classification
results concerning the contact structures
in question see \cite{lutz77,gego95,giro01,kela21}.

\begin{thm}
\label{thm:seifert}
Let $M$ be a closed, oriented $3$-manifold with a fixed-point
free $S^1$-action (in other words, $M$ is a Seifert manifold
with oriented fibres and base)
and a contact structure $\xi$ invariant under the
$S^1$-action. Then $\xi$ admits a Bott-integrable Reeb flow.
\end{thm}

Another approach is to study the existence of Bott-integrable
Reeb flows on $3$-manifolds
for which the classification of contact structures is known.
Here we consider the $3$-sphere, the $3$-torus,
and $S^1\times S^2$.

\begin{thm}
\label{thm:S3T3}
Every contact structure on the $3$-sphere $S^3$, on the $3$-torus $T^3$,
and on $S^1\times S^2$, admits a Bott-integrable Reeb flow.
\end{thm}

For the $3$-sphere, the proof of this theorem will be completed in
Section~\ref{subsection:S3}; for the $3$-torus, in Section~\ref{section:T3};
for $S^1\times S^2$, in Section~\ref{section:S1S2}.
Along the way, we develop topological constructions for integrable
Reeb flows, such as connected sums, gluing along torus boundaries,
or methods related to open books, all of which should prove useful in
wider contexts.

The arguments we use to prove Theorem~\ref{thm:S3T3} can be applied
to other manifolds. For instance, $\RP^3$
admits a unique tight contact structure, and then one
argues as in the proof for $S^1\times S^2$ that every contact structure
admits a Bott-integrable Reeb flow. This reasoning also yields partial
results about the contact structures on lens spaces admitting
Bott-integrable Reeb flows. We plan to address this
systematically in a future publication.

Together with the result of Paternain mentioned before
Proposition~\ref{prop:no-bott}, Theorem~\ref{thm:S3T3} shows that every
contact structure on $S^3$ admits a contact form whose Reeb flow
has zero topological entropy; this answers a question raised
by C\^ot\'e~\cite{cote20}.

Finally, in Section~\ref{section:klein} we present in some detail
examples of Bott-integrable Reeb flows where the critical set
of the Bott integral contains a Klein bottle. This is a rare and
non-generic phenomenon, and we describe how, by a small perturbation
of the Morse--Bott function, one can obtain a function having
only isolated critical Reeb orbits. In the $4$-dimensional
Hamiltonian setting, such genericity and perturbation
results have been obtained by Kalashnikov~\cite{kala95}.
\subsection{Further non-existence statements}
\label{subsection:non-exist}
In this section we collect further examples of graph manifolds carrying
contact structures that do not admit Bott-integrable Reeb flows,
and we comment on the situation in higher dimensions.

In \cite{foha13}, Foulon and Hasselblatt describe Anosov Reeb flows
on the Handel--Thurston manifolds~\cite{hath80}. These are
graph manifolds, and `most' of them are non-trivial, in the sense
that they are not finitely covered
by a Seifert fibred manifold.
The contact structures supporting these Anosov Reeb flows do not
admit Bott-integrable Reeb flows. This follows from
the work of Alves~\cite[Corollary~3]{alve16-pos}
(in combination with the result
of Paternain~\cite{pate91} cited earlier);
Alves shows that if a contact structure on a closed $3$-manifold admits
an Anosov Reeb flow, then all Reeb flows of this contact structure
have positive topological entropy.

The Handel--Thurston manifolds are obtained by a surgery construction,
and Foulon--Hasselblatt showed that these surgeries can be
performed as contact Dehn surgeries.
These ideas have been expanded in \cite{alve16-cyl} and
\cite{fhv21}, where contact homology is used to analyse the
complexity of the Reeb flows resulting from the surgery.

Finally, a few words about the situation in higher dimensions.
There is a notion of non-degeneracy of completely integrable
Hamiltonian flows in all dimensions \cite[Definition~2.1]{pate94},
which in the $3$-dimensional Reeb case translates into
the existence of a Bott integral whose critical set
consists exclusively of periodic orbits. For instance, the examples
of Bott-integrable Reeb flows we construct on the $3$-sphere
when we prove Theorem~\ref{thm:S3T3} are of this type.

Reeb flows in higher
dimensions that are completely integrable with non-degen\-erate
first integrals in the sense just mentioned have zero topological
entropy \cite[Theorem~2.2]{pate94}. On the other hand,
on all spheres of dimension $2n+1\geq 5$ there are contact structures 
all of whose Reeb flows have positive topological entropy;
in dimensions $\geq 7$ this was shown by Alves and Meiwes~\cite{alme19},
and in dimension $5$ by C\^ot\'e~\cite{cote20}. Thus, at least under
this non-degeneracy assumption, an analogue
of Theorem~\ref{thm:S3T3} does not hold for higher-dimensional
spheres.
\subsection{The literature on contact integrability}
A comprehensive survey of the literature on contact
integrability can be found in \cite[Section~3.5]{khta10}.
Most of these studies are concerned with more restrictive
notions of contact integrability, or with higher-dimensional
phenomena.

For instance, some of the earliest work in the field,
by Banyaga and Molino~\cite{bamo93}, deals with
completely integrable contact forms of toric type.
In dimension~$3$, such contact toric manifolds have
been classified by Lerman~\cite{lerm03}. See
also~\cite{boye11} for a general discussion.

Miranda~\cite{mira14} gives a nice unified approach to
integrable systems in symplectic, Poisson and contact manifolds;
in the contact case she assumes that the Reeb vector field generates an
$S^1$-action.

Some other recent papers on the subject of contact integrability
are \cite{jova12,jojo15,visi17}.
\section{Bott integrability}
\label{section:bott-int}
We begin by motivating the definition of integrability for
Reeb flows. We then establish some basic
properties of Bott-integrable Reeb flows, including
the analogue of Liouville's theorem in Hamiltonian dynamics
(for which we give a direct $3$-dimensional proof),
and a neighbourhood theorem for critical Reeb orbits.
\subsection{Motivating the definition}
Consider the symplectisation
\[ \bigl(\R\times M,\omega:=\rmd(\rme^t\alpha)\bigr) \]
of $(M,\alpha)$, where $M$ is a closed, oriented $3$-manifold
and $\alpha$ a positive contact form on~$M$. Our sign convention
for defining the Hamiltonian vector field $X_H$ of a smooth function
$H\co \R\times M\rightarrow\R$ is to require
\[ -\rmd H=\omega(X_H,\,.\,).\]
Thus, for $H(t,p)=\rme^t$ we have $X_H=R_{\alpha}=:R$, since
\[ -\rme^t\,\rmd t=i_R\rme^t(\rmd t\wedge\alpha+\rmd\alpha).\]
We identify vector fields and differential forms on $M$ with
their $\R$-invariant extensions to $\R\times M$.

Given a function $f\co M\rightarrow\R$, we define its extension $F$
to $\R\times M$ by
\[ \begin{array}{rccc}
F\co & \R\times M & \longrightarrow & \R \\
     & (t,p)      & \longmapsto     & \rme^tf(p).
\end{array} \]
Then $\rmd F=\rme^t(\rmd f+f\,\rmd t)$, and with the ansatz
$X_F=a\partial_t+bR+Y$, where $Y(t,p)\in\ker\alpha\subset
TM\equiv T_{t,p}(\{t\}\times M)$, one finds
\[ X_F= -\rmd f(R)\partial_t+fR+Y,\]
with $Y\in\ker\alpha$ defined by
\[ i_Y\rmd\alpha=-\rmd f + \rmd f(R)\alpha.\]
Notice that $X_F$ is $\R$-invariant, and
its projection to $M$ equals the contact Hamiltonian
vector field $X_f=fR+Y$ with respect to the contact form $\alpha$;
see \cite[Theorem~2.3.1]{geig08}.

We compute (with $H(t,p)=\rme^t$ and $X_H=R$)
\[ \omega(R,X_F)=-\rmd H(X_F)=-\rme^t\,\rmd t(X_F)
=-\rme^t\,\rmd f(R).\]
For an arbitrary function $F$ on $(\R\times M,\omega)$,
the key condition for the pair of functions $H,F$
to turn $X_H=R$ into a Liouville integrable system in the sense of
\cite[Definition~1.10]{bofo04} is that $F$ Poisson commute with~$H$,
which means that $\rmd F(X_H)=0$ or, equivalently $\omega(X_H,X_F)=0$.
So for a function $F$ of the form $F=\rme^tf$ this condition becomes
$\rmd f(R)=0$. Then the defining equations for $Y$ simplify to
\begin{equation}
\label{eqn:Y}
\alpha(Y)=0,\;\;\;i_Y\rmd\alpha=-\rmd f,
\end{equation}
and $X_F=fR+Y$ coincides with $X_f$. Notice that in this
situation both $R$ and $Y$ are tangent to the level sets
of~$f$, and $Y$ is non-zero along regular level sets. This implies
that the components of regular level sets are $2$-tori, and the
components of critical level sets (under the Morse--Bott assumption on~$f$)
are periodic orbits of~$R$, tori or Klein bottles.

One also needs to ensure the functional independence
of $H$ and $F$: this translates into the requirement that $\rmd f$
be non-zero almost everywhere. Completeness
of the vector fields $X_H,X_F$ is guaranteed by $M$ being closed.
This justifies our Definition~\ref{defn:reeb-bott-int}.

We now show that the vector fields $R$ and $X_f=fR+Y$ commute.
Since $\rmd f(R)=0$, this is the same as saying that $R$ and
$Y$ commute.

\begin{lem}
In the integrable situation, i.e.\ when $\rmd f(R)=0$,
the vector field $Y$ defined by \eqref{eqn:Y} commutes with the Reeb
vector field: $[R,Y]=0$.
\end{lem}

\begin{proof}
(i) We first show that $[R,Y]$ is tangent to $\ker\alpha$:
\begin{eqnarray*}
0 & = & \rmd\alpha(R,Y)\\
  & = & R\bigl(\alpha(Y)\bigr)-Y\bigl(\alpha(R)\bigr)
        -\alpha([R,Y])\\
  & = & -\alpha([R,Y]).
\end{eqnarray*}

(ii) It remains to show that $\rmd\alpha([R,Y],\,.\,)$
vanishes identically. Using the fact that $i_{[R,Y]}=
[L_R,i_Y]$ as operators on differential forms, we find
\begin{eqnarray*}
\rmd\alpha([R,Y],\,.\,)
  & = & [L_R,i_Y]\rmd\alpha\\
  & = & L_Ri_Y\rmd\alpha\\
  & = & -L_R\rmd f = 0.\qed
\end{eqnarray*}
\renewcommand{\qed}{}
\end{proof}
\subsection{Liouville theorems}
Here is the analogue of the Liouville theorem \cite[Theorem~1.2]{bofo04}
for integrable Reeb flows. For this theorem one only needs
the (Liouville) integrability condition $\rmd f(R)=0$, not that
$f$ is a Morse--Bott function.

Recall that a \emph{pre-Lagrangian} surface in a closed $3$-dimensional
contact manifold $(M,\xi)$ is an embedded surface $\Sigma\subset M$ such that
\begin{itemize}
\item[(i)] $\Sigma$ is transverse to~$\xi$;
\item[(ii)] the line distribution $\xi\cap T\Sigma$
can be defined by a closed $1$-form on~$\Sigma$.
\end{itemize}
This concept (also in higher dimensions) goes back to Bennequin,
see~\cite{ehs95}.

If $\Sigma$ is orientable, it is necessarily diffeomorphic to
a $2$-torus~$T^2$. Using the ideas and results of \cite[Section~9.3]{conl01}
one sees that $T^2\subset (M,\xi)$ is pre-Lagrangian if and only
if the characteristic foliation defined by the line distribution
$\xi\cap T(T^2)$ is diffeomorphic to a linear foliation.
In our situation, where the pre-Lagrangian tori arise as
regular level surfaces of the integral~$f$, the proof of the
following theorem includes a construction of this linearising
diffeomorphism.

\begin{thm}[`Reeb--Liouville']
\label{thm:reeb-liouville}
Let $f$ be an integral for the Reeb flow of $(M,\alpha)$, and let
$\Sigma\subset M$ be a component of a regular level set of~$f$. Then
$\Sigma$ is a closed, oriented embedded surface, and the
following statements hold:

\begin{itemize}
\item[(a)] $\Sigma$ is a pre-Lagrangian torus invariant under the
flow of $R$ and $X_f=fR+Y$.
\item[(b)] A neighbourhood of $\Sigma$ in $M$ is diffeomorphic
to $[-1,1]\times T^2$ such that with coordinates $r$ on $[-1,1]$ and $x_1,x_2$
on $T^2=(\R/2\pi\Z)^2$ we have
\begin{itemize}
\item[(i)] $f=f(r)$ with $\partial f/\partial r>0$ (so that, conversely,
$r$ is a function of~$f$);
\item[(ii)] $\alpha=h_1(r)\,\rmd x_1+h_2(r)\,\rmd x_2$.
\end{itemize}
\end{itemize}
\end{thm}

\begin{rem}
\label{rem:lutz-form}
The contact condition for a $1$-form $\alpha$ as in (b-ii) of the
Reeb--Liouville theorem, with orientation defined by the
volume form $\rmd r\wedge\rmd x_1\wedge\rmd x_2$, becomes
\begin{equation}
\label{eqn:Delta}
\Delta:=\left|\begin{array}{cc}
h_1 & h_1'\\
h_2 & h_2'
\end{array}\right|<0.
\end{equation}
Geometrically this means that the curve $r\mapsto (h_1(r),h_2(r))$
has its trace in $\R^2\setminus\{(0,0)\}$, and the velocity vector
$(h_1'(r),h_2'(r))$ always points to the right of the position
vector $(h_1(r),h_2(r))$.
We call contact forms on $I\times T^2$ (where $I$ can be
any real interval) with such a coordinate description
\textbf{Lutz forms}, because of their role in the Lutz twist
\cite[Section~4.3]{geig08}. Indeed, such Lutz twist will be instrumental
in proving Theorems \ref{thm:bott-graph} and~\ref{thm:liouville-eq}.
Notice that the Reeb vector field of a Lutz form is given by
\[ R=\frac{h_2'\partial_{x_1}-h_1'\partial_{x_2}}{\Delta};\]
in particular, it is tangent to the $T^2$-factor.
\end{rem}

\begin{proof}[Proof of Theorem~\ref{thm:reeb-liouville}]
(a) We have the integrability condition $\rmd f(R)=0$, and thus,
as observed earlier, $Y\neq 0$ along $\Sigma$ and
$\rmd f(Y)=0$ from the defining equation \eqref{eqn:Y} for~$Y$.
Write $i\co\Sigma\rightarrow M$ for the inclusion map. Then
\[ \langle Y\rangle=\ker\alpha\cap T\Sigma=\ker i^*\alpha,\]
the kernel of $i^*\alpha$ being $1$-dimensional since $R$ is tangent
to~$\Sigma$. Again by \eqref{eqn:Y}, the $1$-form $i^*\alpha$
is closed.

(b) Write $\phi_t^R,\phi_t^Y$ for the flow of $R$ and~$Y$, respectively.
Since $R$ and $Y$ commute, we can define an $\R^2$-action on $M$ by
\[ \begin{array}{ccc}
\R^2\times M       & \longrightarrow & M\\
\bigl((t_1,t_2),p) & \longmapsto     & \Phi(t_1,t_2)(p):=
                                       \phi_{t_1}^R\phi_{t_2}^Y(p).
\end{array}\]
Since $R$ and $Y$ are pointwise linearly independent along
the regular level set~$\Sigma$,
for $p_0\in\Sigma$ this defines a covering map
\[ \begin{array}{rccc}
\Phi_0\co & \R^2      & \longrightarrow & \Sigma\\
          & (t_1,t_2) & \longmapsto     & \Phi(t_1,t_2)(p_0).
\end{array} \]
Hence, there is a lattice $\langle e_1^0,e_2^0\rangle$ in $\R^2$
such that $\Phi_0$ descends to a diffeomorphism
$\R^2/\langle e_1^0,e_2^0\rangle\rightarrow T^2\cong\Sigma$;
see the proof of \cite[Lemma~1.4]{bofo04} for further details.
The $1$-form $\rmd f$ defines a coorientation of $\Sigma$,
and with the orientation of $M$ (for which $\alpha\wedge\rmd\alpha$
is a positive volume form) this defines an orientation of~$\Sigma$.
We choose the orientation of the basis $e_1^0,e_2^0$ compatibly with
this orientation.

Let $\nabla f$ be the gradient of $f$ (in terms of some
auxiliary Riemannian metric on~$M$). Then the flow of
$\nabla f/|\nabla f|^2$ (defined near $\Sigma$) preserves the
foliation made up of the
level sets of~$f$. After rescaling this vector field by a small
positive constant, we may assume that the flow is defined
for all $t\in [-1,1]$, and hence defines an identification
of a neighbourhood of $\Sigma$ with $[-1,1]\times T^2$ such that
the tori $\{r\}\times T^2$ are regular level sets of~$f$.
This proves (b-i).

Now, for any $r\in[-1,1]$ we find a lattice $\langle e_1^r,e_2^r\rangle$
such that the covering map
\[ \begin{array}{rccc}
\Phi_r\co & \R^2      & \longrightarrow & \{r\}\times T^2\\
          & (t_1,t_2) & \longmapsto     & \Phi(t_1,t_2)(r,p_0)
\end{array} \]
descends to a diffeomorphism 
$\R^2/\langle e_1^r,e_2^r\rangle\rightarrow \{r\}\times T^2$.
The $e_i^r$ are solutions of the equation
\[ \Phi(t_1,t_2)(r,p_0)=(r,p_0),\]
and since $\partial\Phi/\partial t_1=R$ and $\partial\Phi/\partial t_2=Y$
are pointwise linearly independent, the $e_i^r$ depend smoothly on $r$
by the implicit function theorem.

By writing $(r,p)\in\{r\}\times T^2$ as
\[ \Phi\bigl(x_1e_1^r/2\pi+x_2e_2^r/2\pi\bigr)(r,p_0)\]
we define angular coordinates $(x_1,x_2)$ on $T^2=(\R/2\pi\Z)^2$.
This finishes the construction of local coordinates, and it remains to
show that in terms of these coordinates
(and, as we shall see, after a further diffeomorphism),
$\alpha$ is a Lutz form.

By construction we have
\begin{eqnarray*}
R & = & a_1\partial_{x_1}+a_2\partial_{x_2},\\
Y & = & b_1\partial_{x_1}+b_2\partial_{x_2},
\end{eqnarray*}
with smooth functions $a_1,a_2,b_1,b_2$ depending on $r$ only.
Vice versa, the $\partial_{x_i}$ can be written as pointwise linear
combinations of $R$ and $Y$ with coefficients depending on $r$ only.

Write the contact form $\alpha$ on $[-1,1]\times T^2$ as
\[ \alpha=h_0\,\rmd r+ h_1\,\rmd x_1+h_2\,\rmd x_2.\]
Then from $h_i=\alpha(\partial_{x_i})$ for $i=1,2$, we see
that $h_1=h_1(r)$ and $h_2=h_2(r)$.

The $\rmd r$-component of the equation $i_R\rmd\alpha=0$ reads
\[ a_1\frac{\partial h_0}{\partial x_1}
+a_2\frac{\partial h_0}{\partial x_2}
-a_1\frac{\partial h_1}{\partial r}
-a_2\frac{\partial h_2}{\partial r} =0.\]
This implies that $\rmd h_0(R)$
depends on $r$ only. Likewise, the equation
$i_Y\rmd\alpha=-\rmd f$ translates into
\[ b_1\frac{\partial h_0}{\partial x_1}+b_2\frac{\partial h_0}{\partial x_2}
 - b_1\frac{\partial h_1}{\partial r} -b_2\frac{\partial h_2}{\partial r}
=-f'(r),\]
so we see that $\rmd h_0(Y)$ also depends on $r$ only. Since $R$ and
$Y$ are pointwise linearly independent, and $h_0$ is
$2\pi$-periodic in $x_1,x_2$, this forces $h_0=h_0(r)$. In particular,
$\rmd\alpha$ now simplifies to
\[ \rmd\alpha=h_1'(r)\,\rmd r\wedge\rmd x_1+h_2'(r)\,\rmd r\wedge\rmd x_2.\]

It remains to get rid of the term $h_0(r)\,\rmd r$ by a suitable
diffeomorphism of $[-1,1]\times T^2$. This can be done by a Gray
deformation \cite[Section~2.2]{geig08}
as follows. Consider the $1$-parametric family of contact forms
\[ \alpha_t:=th_0(r)\,\rmd r+h_1(r)\,\rmd x_1+h_2(r)\,\rmd x_2,\;\;\;
t\in [0,1].\]
Notice that the $\alpha_t$ all share the Reeb vector field~$R$.
On a closed $3$-manifold this is sufficient to guarantee that the
contact forms in this family are diffeomorphic~\cite[Proposition~2.1]{geig22};
in the present situation of a manifold with boundary
we need to verify that the vector field provided by
the Moser trick integrates up to time~$1$.

When we apply the Moser trick~\cite[p.~60]{geig08} to the equation
\begin{equation}
\label{eqn:moser}
\psi_t^*\alpha_t=\alpha_0,
\end{equation}
assuming that $\psi_t$ is the flow of a time-dependent vector
field $X_t\in\ker\alpha_t$, by differentiating \eqref{eqn:moser}
with respect to $t$ we obtain
\[ \dot{\alpha}_t+i_{X_t}\rmd\alpha_t=0.\]
With the time-independent ansatz
$X=c_1(r)\partial_{x_1}+c_2(r)\partial_{x_2}$,
the conditions on $X=X_t$ translate into
\[ \begin{pmatrix}
h_1  & h_2\\
h_1' & h_2'
\end{pmatrix}
\begin{pmatrix}
c_1\\
c_2
\end{pmatrix}=
\begin{pmatrix}
0\\
h_0
\end{pmatrix}, \]
which has a unique solution $(c_1(r),c_2(r))$ by the contact
condition~\eqref{eqn:Delta}. Since $X$ is tangent to the
$T^2$-factor, it integrates up to time~$1$ to yield the desired isotopy.
\end{proof}

By a variant of this argument we can also establish a normal
form for the neighbourhood of an \textbf{elliptic orbit}
in the Reeb flow, that is, a periodic orbit along which the Bott integral $f$
has a minimum or maximum in transverse direction, so that the
orbit constitutes a connected component of the critical set $\Crit(f)$.
Here the \emph{Bott} integrability is essential.

\begin{thm}[Neighbourhood theorem for elliptic Reeb orbits]
\label{thm:nbhd}
Let $f$ be a Bott integral for the Reeb flow of $(M,\alpha)$, and let
$\gamma\subset M$ be a periodic Reeb orbit transverse to which
$f$ has a local minimum. Then there are coordinates $(\theta,r,\varphi)$
on a neighbourhood $S^1\times D^2_{\delta}$ of $\gamma$,
with $\gamma=S^1\times\{0\}$, such that $f=r^2$ on that
neighbourhood, and a contact form
\[ \alpha'=h_0(r)\,\rmd r+h_1(r)\,\rmd\theta+h_2(r)\,\rmd\varphi\]
on $S^1\times D^2_{\delta}$ such that
\begin{itemize}
\item[(i)] $\alpha'$ coincides with $\alpha$ near $S^1\times\partial
D^2_{\delta}$;
\item[(ii)] on a smaller neighbourhood of $\gamma$ we have
$\alpha'=\rmd\theta+r^2\,\rmd\varphi$;
\item[(iii)] the contact forms $\alpha$ and $\alpha'$
on $S^1\times D^2_{\delta}$ are isotopic relative to a neighbourhood
of the boundary via contact forms all having the same Bott integral.
\end{itemize}
In particular, the contact structure on $M$ defined by $\alpha'$ on $S^1\times
D^2_{\delta}$ and $\alpha$ on the complement of this solid torus
is isotopic, by Gray stability, to the one defined by $\alpha$
on all of~$M$.
\end{thm}

\begin{proof}
The generalised Morse--Bott lemma \cite[Lemma~1.7]{bofo04}
provides us with a neighbourhood $S^1\times D^2_{\delta}$
of $\gamma=S^1\times\{0\}$ where $f=r^2$ in terms of a transverse radial
coordinate~$r$. Choose a smooth path $[0,\delta]\ni r
\mapsto p_0(r)\in S^1\times D^2_{\delta}$ with $f(p_0(r))=r$. Using
the flows $\phi_t^R$ and $\phi_t^Y$ as in the proof of
Theorem~\ref{thm:reeb-liouville}, we define
\[ \begin{array}{rccc}
\Phi\co & [0,\delta]\times\R^2 & \longrightarrow & S^1\times D^2_{\delta}\\
        & (r,t_1,t_2)          & \longmapsto     & \phi_{t_1}^R\phi_{t_2}^Y
                                                   (p_0(r)).
\end{array}\]
We then find a smooth family of lattices $\langle e_1^r,e_2^r\rangle$,
$r\in (0,\delta]$, determined by the following conditions:
\begin{itemize}
\item[-] The map $\Phi$ descends to a diffeomorphism
\[ \{r\}\times\R^2/\langle e_1^r,e_2^r\rangle\longrightarrow
S^1\times S^1_r;\]
\item[-] the loop  $t\mapsto t e_1^r$ on the left-hand side,
$t\in [0,1]$, maps to a loop on the torus $S^1\times S^1_r$
isotopic to $S^1\times *$;
\item[-] the loop  $t\mapsto t e_2^r$ maps to a loop isotopic to
$*\times S^1_r$.
\end{itemize}
We define coordinates $(\theta,r,\varphi)$ on $S^1\times\bigl(D^2_{\delta}
\setminus\{0\}\bigr)$ by the parametrisation
\[ (\theta,r,\varphi)\longmapsto \Phi(\theta e_1^r/2\pi+\varphi e_2^r/2\pi)
(p_0(r)).\]
As in the preceding proof one then sees that
\[ \alpha = h_0(r)\,\rmd r+h_1(r)\,\rmd\theta+h_2(r)\,\rmd\varphi\;\;\;
\text{on $S^1\times\bigl(D^2_{\delta}\setminus\{0\}\bigr)$}.\]
Since we know \emph{a priori} that $\alpha$ extends smoothly
over $S^1\times\{0\}=\gamma$, this forces $\alpha$ to coincide with
some constant positive multiple of $\rmd\theta$ along~$\gamma$, with
$\theta$ defining
a parametrisation of~$\gamma$. We now modify $\alpha$ near $\gamma$ in
several steps.

\vspace{1mm}

Step 1: Replace $h_0(r)$ by $\psi(r)h_0(r)$, where
$\psi\co [0,\delta]\rightarrow [0,1]$ is a smooth function identically equal
to $0$ near $r=0$, say on $[0,\delta_1]$,
and identically equal to $1$ near~$r=\delta$.

This leaves the Reeb vector field unchanged, so $f$ is still an integral.
Moreover, the linear deformation from $\alpha$ to the new contact form
is via contact forms. 

\vspace{1mm}

Step 2: We may assume that $\delta_1>0$ has been chosen sufficiently
small such that $h_1>0$ on $[0,\delta_1]$. Now choose a smooth
function $\chi\co[0,\delta_1]\rightarrow\R^+$ with $\chi(r)=h_1(r)$
near $r=0$, say on $[0,\delta_2]$, and $\chi(r)=1$ near $\delta_1$.
Then replace the contact form $h_1(r)\,\rmd\theta+h_2(r)\rmd\varphi$
on $S^1\times D^2_{\delta_1}$ by
\[ \frac{1}{\chi(r)}\bigl(h_1(r)\,\rmd\theta+h_2(r)\,\rmd\varphi\bigr).\]
This does not change the contact structure, but it does affect the Reeb
vector field. However, by Remark~\ref{rem:lutz-form}, the Reeb vector
field stays tangent to the tori $S^1\times S^1_r$, so $f$ is still
an integral.

\vspace{1mm}

Step 3: Writing again $h_2$ for $h_2/\chi$, the new contact form
on $S^1\times D^2_{\delta_2}$ is $\rmd\theta+h_2(r)\,\rmd\varphi$.
Notice that the contact condition translates into $h_2'>0$, and
the $1$-form can only be smooth in $r=0$ if $h_2(0)=0$. We can therefore
find a smooth function $h_2^*\co[0,\delta_2]\rightarrow\R_0^+$
that coincides with $h_2$ near $r=\delta_2$, and near $r=0$ is
given by $h_2^*(r)=r^2$. The linear deformation
\[ \rmd\theta+\bigl((1-t)h_2(r)+th_2^*(r)\bigr)\,\rmd\varphi,\;\;\;
t\in[0,1],\]
is supported in a neighbourhood of $\gamma$ and
via contact forms of Lutz type.
\end{proof}

For critical Reeb orbits of \textbf{hyperbolic} type, where the Bott
integral has a saddle point in transverse direction, there are two local
models for the Liouville foliation, depending on the separatrix
diagram being orientable or not; see~\cite[Section~3.1]{bofo04}.
We shall not try to formulate a normal form theorem for
hyperbolic orbits, as it is irrelevant for our further discussion.
\subsection{Creating isolated critical Reeb orbits}
\label{subsection:create-orbits}
For the connected sum construction in Section~\ref{subsection:cont-connect}
we require that the Bott-integrable Reeb flow on either
summand has an elliptic Reeb orbit. In general, such orbits need not exist.

\begin{ex}
On $T^3=S^1\times S^1\times S^1$ with circular coordinates $x,y,z$
we consider the contact form $\alpha=\cos z\,\rmd x-\sin z\,\rmd y$
with Reeb vector field $R=\cos z\,\partial_x-\sin z\,\partial_y$.
As Bott function we take $f(x,y,z):=\cos z$. This is a Morse function
on the $z$-circle, and hence a Morse--Bott function on $T^3$, with critical
set
\[ \Crit(f)=\{z=0\}\sqcup\{z=\pi\} \]
consisting of two $2$-tori. Clearly, we have $\rmd f(R)=0$.
\end{ex}

The next proposition says that critical Reeb orbits may be introduced
\emph{ad libitum}, at the cost of changing the contact form and the
Bott function, while leaving the contact structure unchanged up to
isotopy.

\begin{prop}
\label{prop:create-orbit}
If $(M,\alpha)$ is Bott integrable with Bott integral~$f$, then by a
local isotopic modification of $\alpha$ near a regular level set $T^2$ of $f$,
fixing the isotopy class of $\ker\alpha$, one can obtain a new
contact form $\alpha^*$ with a Bott integral $f^*$ such that
\begin{itemize}
\item[(i)] $\Crit(f^*)$ is the union of $\Crit(f)$ and a pair of
periodic orbits of the new Reeb vector field~$R_{\alpha^*}$,
one of elliptic and one of hyperbolic type;
\item[(ii)] $f^*$ coincides with $f$ outside an arbitrarily
small neighbourhood of those additional critical orbits.
\end{itemize}
\end{prop}

\begin{proof}
By Theorem~\ref{thm:reeb-liouville}, it suffices to consider
a Lutz form $\alpha=h_1(r)\,\rmd x_1+h_2(r)\,\rmd x_2$ on $[-1,1]\times T^2$
with non-singular Bott integral $f=f(r)$. By choosing the coordinates
$x_1,x_2$ on $T^2$ appropriately and rescaling
$\alpha$ by a positive constant, we may assume that $h_1'(0)>0$ and
$h_2(0)=1$. One can then easily modify the functions $h_1,h_2$ near
$r=0$ by a homotopy through functions with negative $\Delta$
(as defined in~\eqref{eqn:Delta}) such that $h_1(r)=r$ and
$h_2(r)=1$ near $r=0$, see Figure~\ref{figure:deform-h}.

\begin{figure}[h]
\labellist
\small\hair 2pt
\pinlabel $h_1$ [t] at 321 71
\pinlabel $h_2$ [r] at 109 282
\endlabellist
\centering
\includegraphics[scale=0.5]{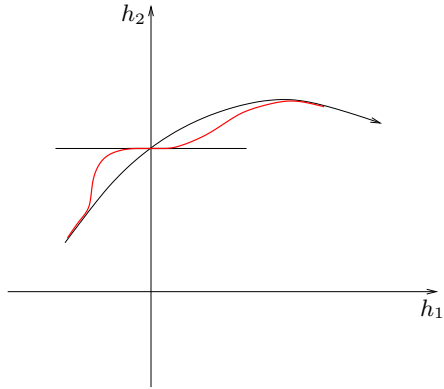}
  \caption{Modifying the Lutz form into $\rmd x_2+r\,\rmd x_1$.}
  \label{figure:deform-h}
\end{figure}

After this modification, the Reeb vector field equals $\partial_{x_2}$
near $\{0\}\times T^2$, say on $[-\varepsilon,\varepsilon]\times T^2$.
Any Morse function on $[-\varepsilon,\varepsilon]
\times S^1_{x_1}$ that coincides with $f(r)$ near $r=\pm\varepsilon$ will
lift to a $\partial_{x_2}$-invariant Morse--Bott function on
this thickened torus and hence define a Bott integral for the
Reeb flow of the modified contact form. The level sets of such a
Morse function $f^*$ with a new local minimum are shown in
Figure~\ref{figure:new-morse}. This $f^*$ may be chosen to
differ from $f$ in a small neighbourhood of $(r,x_1)=(0,0)$ only.

\begin{figure}[h]
\labellist
\small\hair 2pt
\pinlabel $r$ [t] at 426 216
\pinlabel $x_1$ [r] at 216 427
\endlabellist
\centering
\includegraphics[scale=0.4]{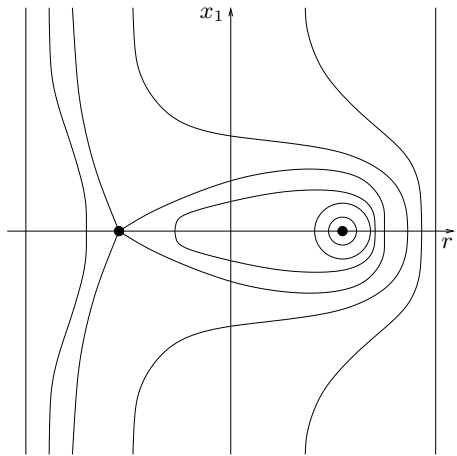}
  \caption{The modified Bott integral $f^*$.}
  \label{figure:new-morse}
\end{figure}
\end{proof}
\section{$3$-manifolds admitting integrable Reeb flows}
In this section we are going to prove Theorems \ref{thm:bott-graph}
and~\ref{thm:liouville-eq}. We first recall the topological
background, in particular the definition of graph manifolds. Next we
construct model Reeb flows on some relevant building blocks,
and then show how these building blocks can be glued along their
boundaries.
\subsection{Graph manifolds}
\label{subsection:graph-manifold}
Graph manifolds have been invented by Waldhausen~\cite{wald67I,wald67II}.
They are the closed, orientable $3$-manifolds that can be decomposed
along a disjoint collection of embedded $2$-tori into $S^1$-fibred
pieces. Equivalently, and more efficiently, one may require that the
pieces merely be Seifert fibred.

It is implicit in Waldhausen's papers that this class of manifolds is
closed under the connected sum operation. Since one can always cut out
an $S^1$-fibred solid torus from a given $S^1$-bundle, and one can
perform the connected sum along two $3$-balls inside such solid tori,
this connected sum result follows from the next lemma. Here by
\emph{graph manifold with boundary} we mean the obvious extension
of the class of graph manifolds (which are closed by definition) to
compact manifolds where each boundary component is a $2$-torus
foliated by $S^1$-fibres.

\begin{lem}
The connected sum of two solid tori is a graph manifold with boundary.
\end{lem}

\begin{proof}
Think of the two solid tori as each being embedded in a copy of $S^3$
in the standard way, with complement another solid torus. Thus, the connected
sum of two solid tori equals the connected sum of two $3$-spheres,
i.e.\ another $3$-sphere, with two unknotted and unlinked solid tori
removed. This, in turn, equals a solid torus with a small unknotted
solid torus --- by `small' we mean contained in a ball ---
removed from its interior.

Thus, we need to show that
\[ V:= (S^1\times D^2)\setminus\nu K,\]
with $\nu K$ an open tubular neighbourhood of $K:=*\times S^1_{1/2}$, say,
is a graph manifold.
The idea for the following construction is taken from~\cite{soma81}.
Cut $V$ along the $2$-torus
$S^1\times S^1_{1/4}$. This separates $V$ into a solid torus
$V_1:=S^1\times D^2_{1/4}$, which fibres in the obvious way, and the
complementary piece
\[ V_2:=\bigl(S^1\times (D^2\setminus\Int(D^2_{1/4}))\bigr)\setminus\nu K.\]
Now, $\bigl(S^1\times (D^2\setminus\Int(D^2_{1/4}))$ is fibred
by concentric circles in the second factor, with quotient equal to
$S^1\times [1/4,1]$. The circle $K$ is one of these fibres, and we
may take $\nu K$ to be a fibred neighbourhood of~$K$. Then $V_2$
is an $S^1$-bundle over $S^1\times [1/4,1]$ with an open disc around
$*\times\{1/2\}$ removed.
\end{proof}

It is also implied by Waldhausen's work that the summands in
a prime decomposition of a graph manifold are likewise graph manifolds,
cf.\ \cite[Corollary 2.7]{yano85}.

The following equivalent description of graph manifolds has been
established in \cite[Section~4.9.6]{bofo04}.

\begin{prop}
\label{prop:graph-AB}
The class of graph manifolds coincides with the closed oriented $3$-manifolds
that can be obtained by gluing finitely many copies of a solid torus
$A:=S^1\times D^2$ and copies of $B:=S^1\times(\mbox{\rm pair of pants})$
along their torus boundaries.\qed
\end{prop}

Here is the idea of the proof.
It is clear that every manifold with an $(A,B)$-decomposition as described
is a graph manifold. Conversely, one needs to show that
every Seifert fibration over a compact surface has an $(A,B)$-decomposition.
For this one first cuts out solid tori around the singular fibres or ---
if there are no singular fibres --- one solid torus around a regular fibre,
so that the base $\Sigma$ of the fibration has non-empty boundary.
If $\Sigma$ is oriented, the $S^1$-fibration is trivial and can obviously
be obtained by gluing copies of $A$ and~$B$. If the base is non-orientable,
the fibration must restrict to the unique non-trivial $S^1$-bundle over
each M\"obius band contained in~$\Sigma$
(since the total space is orientable). The total space of this bundle has
boundary $T^2$, so we can cut these pieces from the bundle over~$\Sigma$.
Finally, it only remains to observe that the non-trivial
$S^1$-bundle over the M\"obius band has an alternative Seifert fibration
over the disc with two singular fibres of multiplicity~$2$
\cite[Lemma~4.9]{bofo04}; see also \cite[Lemma~4.4]{geth}.

\subsection{Integrable Reeb flows on the building blocks}
\label{subsection:flow-models}
For the proof of Theorem~\ref{thm:bott-graph} we only need the
building blocks~$A$ and~$B$; for the proof of Theorem~\ref{thm:liouville-eq}
we shall have to deal with more general Seifert fibred building blocks
coming from involutions on compact surfaces with boundary.
All building blocks are compact $3$-manifolds with torus boundaries, and
the contact forms we describe presently are Lutz forms near those boundaries.
\subsubsection{The solid torus}
On $S^1\times D^2$ with coordinates $(\theta;r,\varphi)$ we take the contact
form $\alpha_A=\rmd\theta+r^2\,\rmd\varphi$. Its Reeb vector field
is $R_A=\partial_{\theta}$, and a Bott integral is given by
$f(\theta;r,\varphi)=c\pm r^2$.
\subsubsection{Pair of pants times $S^1$}
Let $\Sigma$ be a pair of pants. Fix an orientation on~$\Sigma$.
Let $\varphi_i\in\R/2\pi\Z$, be angular coordinates along the three boundary
components $\partial_i\Sigma$, $i=1,2,3$, compatible with the boundary
orientation. Let $r_i\in(-1,0]$ be a collar parameter near $\partial_i\Sigma$,
with $\partial_i\Sigma=\{r_i=0\}$.

We want to find an exact area form $\omega=\rmd\lambda$ on $\Sigma$, with
$\lambda$ looking like $h_i(r_i)\,\rmd\varphi_i$ near $\partial_i\Sigma$,
with $h_i'(r_i)>0$. This can be done geometrically as follows.
Embed $\Sigma$ into $\R^3$ as a pair of carrot pants with a rather large
waistline as shown in Figure~\ref{figure:carrot-pants}. We may assume
that the three ends have conical shape over the respective boundary circle.
Let $\Omega$ be the area form of the metric induced from~$\R^3$,
and let $Y=\nabla f/|\nabla f|^2$ be the normalised gradient
of the height function~$f$, whose flow preserves the $f$-levels.
The carrot shape guarantees that $Y$ has positive divergence with respect
to~$\Omega$, so
\[ \rmd(i_Y\Omega)=\mathrm{div}_{\Omega}(Y)\Omega\]
is an exact area form, and $\Sigma$ having conical ends means that the
primitive $\lambda:=i_Y\Omega$ looks as desired near~$\partial\Sigma$.

On $B=S^1\times\Sigma$ we then take the contact form $\alpha_B=\rmd\theta
+\lambda$, with Reeb vector field $R_B=\partial_{\theta}$
and Bott integral~$f$.

\begin{figure}[h]
\centering
\includegraphics[scale=0.3]{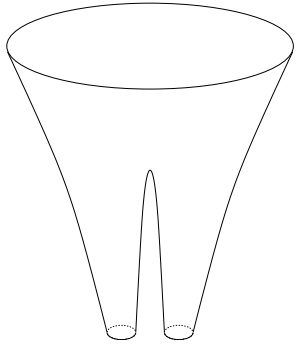}
  \caption{Carrot pants.}
  \label{figure:carrot-pants}
\end{figure}
\subsubsection{Seifert fibred building blocks}
\label{subsubsection:general}
In order to describe Hamiltonian or Reeb flows on $3$-manifolds
up to Liouville equivalence, we need to work with more general
building blocks. We continue to use the building block~$A=S^1\times D^2$, with
Liouville foliation defined by the lift of $f$, i.e.\ the natural foliation
of $S^1\times D^2$ with one singular leaf $S^1\times\{0\}$ and concentric
torus leaves.

Building block $B$ will be subsumed by the following more
general construction.
Consider a surface $\Sigma$ with a Morse function
\[ f\co\Sigma\longrightarrow [c-\varepsilon,c+\varepsilon]\]
mapping surjectively onto $[c-\varepsilon,c+\varepsilon]$,
such that
\[ \partial\Sigma=\{f=c\pm\varepsilon\},\]
and with $c$ being the only potentially critical value of~$f$.
In particular, $\Sigma$ will be of genus~$0$.
This gives rise to a building
block $S^1\times\Sigma$, and as Liouville foliation
with at most one singular leaf we take the one
defined by the lift of $f$ to $S^1\times\Sigma$. Notice that collar
neighbourhoods of the boundary of this $3$-manifold will be foliated
by regular tori.

Further building blocks come from surfaces $\Sigma$ that, in addition,
admit an orien\-ta\-tion-preserving involution $\tau\co\Sigma\rightarrow\Sigma$
with the properties that
\begin{itemize}
\item[(i)] $\tau$ preserves $f$, that is, $f\circ\tau=f$; and
\item[(ii)] $\tau$ has only finitely many isolated fixed points,
all being critical points of~$f$.
\end{itemize}

As an example, you may take the pair of pants in
Figure~\ref{figure:carrot-pants} with $\tau$ the rotation through
an angle $\pi$ about the vertical symmetry axis through the
critical point of the height function.

The mapping torus
\[ M_{\tau}:=[0,\pi]\times\Sigma/(\pi,x)\sim(0,\tau(x))\]
is then a Seifert fibred space over the orbifold quotient $\Sigma/\tau$,
with singular fibres of multiplicity~$2$ corresponding to the
fixed points of~$\tau$. Again, we use the lift of $f$ to define
the Liouville foliation on $M_{\tau}$.

As shown in Chapters 3 and~4 of~\cite{bofo04} (notably
Theorems 4.1 and~4.2), any topologically stable
Hamiltonian flow on a $3$-dimensional energy surface of a
Bott-integrable system is Liouville equivalent to one obtained by
gluing such building blocks.

The construction of a contact form on these general building blocks
is based on the following lemma, which we formulate in a way that is
sufficiently general for other applications further below.

\begin{lem}
\label{lem:exact-area}
Let $\Sigma$ be a compact, oriented surface with non-empty
boundary. Choose collar coordinates $(r_i,\varphi_i)\in (-1,0]
\times\R/2\pi\Z$ near each boundary component $\partial_i\Sigma=
\{r_i=0\}$, where $i$ runs over a finite index set. It is understood
that $\rmd\varphi_i$ defines the orientation $\partial_i\Sigma$
as oriented boundary of~$\Sigma$. Let
\[ \lambda_i=\rho_i(r_i)\,\rmd\varphi_i\]
be $1$-forms on these collar neighbourhoods with $\rho_i'>0$,
so that the $\rmd\lambda_i$ are (positive) area forms on the collars.

If $\sum_i\int_{\partial_i\Sigma}\lambda_i>0$, there is
an exact area form $\omega=\rmd\lambda$ on $\Sigma$ with
$\lambda=\lambda_i$ near $\partial_i\Sigma$.
\end{lem}

\begin{proof}
The integral condition allows us to choose
an area form $\omega$ on $\Sigma$ that coincides with
$\rmd\lambda_i$ near $\partial_i\Sigma$ and satisfies
\[ \int_{\Sigma}\omega=\sum_i\int_{\partial_i\Sigma}\lambda_i.\]
Let $\lambda_{\partial}$ be an extension of the $\lambda_i$ to a global $1$-form
on~$\Sigma$. Then $\omega-\rmd\lambda_{\partial}$ is a $2$-form
compactly supported in the interior of~$\Sigma$, i.e.\
an element of $\Omega^2_{\mathrm{c}}(\Int(\Sigma))$.

From de Rham theory for compactly supported forms one knows that
the sequence
\[ \Omega^1_{\mathrm{c}}(\Int(\Sigma))\stackrel{\rmd}{\longrightarrow}
\Omega^2_{\mathrm{c}}(\Int(\Sigma))
\stackrel{\int_{\Sigma}}{\longrightarrow}\R\]
is exact; see \cite[Theorem~10.13]{mato97} or
\cite[Corollary~5.8]{botu82}. Thus, from
\[ \int_{\Sigma}(\omega-\rmd\lambda_{\partial})=
\int_{\Sigma}\omega-\int_{\partial\Sigma}\lambda_{\partial}=
\int_{\Sigma}\omega-\sum_i\int_{\partial_i\Sigma}\lambda_i=0\]
we conclude that $\omega-\rmd\lambda_{\partial}=\rmd\lambda_{\mathrm{c}}$
for some compactly supported $1$-form $\lambda_{\mathrm{c}}$. Then
$\lambda:=\lambda_{\partial}+\lambda_{\mathrm{c}}$ is a primitive of the
area form~$\omega$ having the desired properties.
\end{proof}

This lemma can easily be adapted to the situation where $\Sigma$
admits an involution $\tau$ as above. Up to isotopy we may
assume that $\tau$ rotates a given collar $(-1,0]\times\partial_i\Sigma$
through an angle~$\pi$, or, if $\tau$ exchanges the collars
of $\partial_i\Sigma$ and $\partial_j\Sigma$, 
that it is the identity map in terms of the collar coordinates
$(r_i,\varphi_i)$ and $(r_j,\varphi_j)$. In the latter case,
one needs to assume that $\rho_i=\rho_j$ in the definition of
$\lambda_i,\lambda_j$. Then the argument goes through as before
and yields a $\tau$-invariant area form $\omega$ with a $\tau$-invariant
primitive $\lambda$ prescribed near the boundary.

It follows that the contact form $\rmd\theta+\lambda$ on
$[0,1]\times\Sigma$ descends to a contact form on the mapping torus
$M_{\tau}$ with Bott integral~$f$.
\subsection{The sewing lemma}
The following is the simple Reeb analogue of the sewing lemma for
integrable $4$-dimensional Hamiltonian systems~\cite[Lemma~4.7]{bofo04}.
It will allow us to glue the building blocks introduced above.

\begin{lem}[Sewing lemma]
\label{lem:sewing}
Let $\alpha=h_1(r)\,\rmd x_1+h_2(r)\,\rmd x_2$ be a Lutz form on
\[ \bigl([-1,-1+\varepsilon]\cup[1-\varepsilon,1]\bigr)\times T^2,\]
for some $\varepsilon\in (0,1)$. Then $\alpha$ extends to a Lutz form
on all of $[-1,1]\times T^2$.
\end{lem}

\begin{proof}
One can always interpolate the curves $r\mapsto (h_1(r),h_2(r))$ defined
on
\[ [-1,-1+\varepsilon]\;\;\;\text{and}\;\;\; [1-\varepsilon,1] \]
subject to the condition $\Delta<0$, as shown in Figure~\ref{figure:sewing}.
\end{proof}

\begin{figure}[h]
\labellist
\small\hair 2pt
\pinlabel $h_1$ [t] at 427 217
\pinlabel $h_2$ [r] at 216 424
\endlabellist
\centering
\includegraphics[scale=0.4]{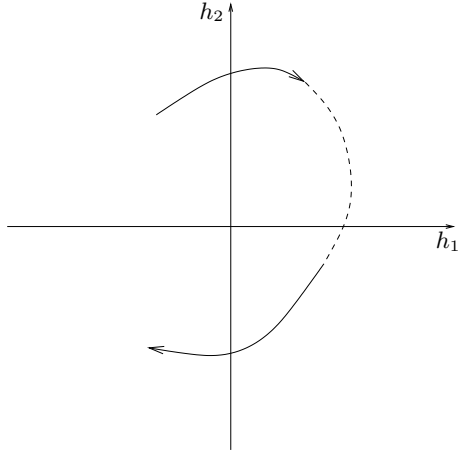}
  \caption{Interpolating Lutz forms.}
  \label{figure:sewing}
\end{figure}
\subsection{Graph manifolds admit integrable Reeb flows}
\label{subsection:thm1}
The topological classification of $3$-manifolds admitting
Bott-integrable Reeb flows is now a straightforward consequence.

\begin{proof}[Proof of Theorem~\ref{thm:bott-graph}]
Given a graph manifold, we present it as a gluing of a finite
number of the building blocks $A$ and~$B$ as in
Proposition~\ref{prop:graph-AB}. The gluing of any two torus
boundaries $T^2_{\pm}$ (oriented as
the boundaries of the corresponding building block) is described by an
orientation-\emph{reversing} diffeomorphism $\phi\co T^2_-\rightarrow T^2_+$.
Up to diffeomorphism of the resulting manifold, the gluing
may be effected by inserting a thickened torus. Explicitly,
we identify collars of  $(-1,1)\times T^2$ with collar neighbourhoods
of $T^2_{\pm}$ via
the orientation-\emph{preserving} diffeomorphisms
\[ \begin{array}{rccc}
\Phi_-\co & (-1,-1+\varepsilon]\times T^2 & \longrightarrow
   & (-1,0]\times T^2_-\\
          & (r,x)                         & \longmapsto
   & \bigl((r+1-\varepsilon)/\varepsilon,x\bigr)
\end{array}\]
and
\[ \begin{array}{rccc}
\Phi_+\co & [1-\varepsilon,1)\times T^2 & \longrightarrow
   & (-1,0]\times T^2_+\\
          & (r,x)                       & \longmapsto
   & \bigl(-(r-1+\varepsilon)/\varepsilon,\phi(x)\bigr).
\end{array}\]
By Section~\ref{subsection:flow-models} we may assume that
the contact forms on the two building blocks we want to glue
are Lutz forms $\alpha_{\pm}$ on the collar neighbourhoods of $T^2_{\pm}$.
Then $\Phi_{\pm}^*\alpha_{\pm}$ are likewise Lutz forms.
By the sewing lemma they extend to $(-1,1)\times T^2$.

On the collar neighbourhoods of $T^2_{\pm}$, the Bott integral
is a non-critical function of the collar parameter only, which under
$\Phi_{\pm}$ pulls back to a function of~$r$. We extend this to
a function on the whole interval $(-1,1)$, possibly with a
single critical point at $r=0$, which would create a critical
$2$-torus in the glued $3$-manifold.
\end{proof}
\subsection{Liouville equivalence}
\label{subsection:liouville-eq}
As mentioned in the introduction, and explained in detail
in the monograph~\cite{bofo04}, any $3$-manifold arising as
an energy hypersurface of a Bott-integrable Hamiltonian flow
has a decomposition --- topologically --- into $A$ and $B$ pieces.
In the preceding section we have seen how to construct
a Bott-integrable Reeb flow from such a decomposition.
However, the Liouville foliation of the Bott integral thus
constructed may not be the Liouville foliation of the
original system. For instance, some of the $A$ pieces may come
from neighbourhoods of hyperbolic critical Reeb orbits,
i.e.\ orbits where the Bott integral is $f(x,y)=c+x^2-y^2$
in terms of transverse cartesian coordinates $(x,y)$. Our construction,
by contrast, will always produce a foliation of the $A$ pieces by
concentric $2$-tori.

This is where the more general building blocks from
Section~\ref{subsubsection:general} come into play, for the
decomposition into these pieces respects the Liouville
foliation of a given Hamiltonian flow, as explained in~\cite{bofo04}.
Now the proof of
Theorem~\ref{thm:bott-graph} applies equally to these general
building blocks and produces a Bott-integrable Reeb flow
with the previously given Liouville foliation. The critical tori
we had to introduce are irrelevant for the Liouville
equivalence, and indeed they should not be required now, since
the gluing of the building blocks is consistent with the
original Bott function, so the monotonicity of the Bott
integral on two respective collars will be respected by the gluing.

This concludes the proof of Theorem~\ref{thm:liouville-eq}.
\section{Integrability of $S^1$-invariant contact structures}
\label{section:S1invariant}
In this section we prove Theorem~\ref{thm:seifert},
but we begin with an example.
\subsection{Liouville--Cartan and connection forms}
Write $\lambda_1$ for the Liouville--Car\-tan form on the unit cotangent
bundle $ST^*\Sigma_g$ of the closed, oriented surface $\Sigma_g$ of
genus~$g$ (with some Riemannian metric chosen on~$\Sigma_g$).
Recall the example from Proposition~\ref{prop:no-bott}:
$\ker\lambda_1$ does not admit a Bott-integrable Reeb flow for
$g\geq 2$. Of course, every $S^1$-bundle
is a graph manifold, so by Theorem~\ref{thm:bott-graph}, $ST^*\Sigma_g$
carries \emph{some} (positive) contact structure that admits a Bott-integrable
Reeb flow.

Recall that there is a second Liouville--Cartan form $\lambda_2$
on $ST^*\Sigma_g$, canonically defined in terms of the Riemannian metric,
such that $\lambda_1\wedge\lambda_2$ is the lift of the (positive)
area form defined by the Riemannian metric and orientation on~$\Sigma_g$.
Together with the Riemannian connection $1$-form~$\alpha$,
the following structure equations are
satisfied, cf.~\cite[Section~2.1]{agz18}. Here $K$ denotes the
Gau{\ss} curvature of the Riemannian metric on~$\Sigma_g$, and
$\pi\co ST^*\Sigma_g\rightarrow\Sigma_g$ the bundle projection:
\begin{eqnarray*}
\rmd\lambda_1 & = & \lambda_2\wedge\alpha,\\
\rmd\lambda_2 & = & \alpha\wedge\lambda_1,\\
\rmd\alpha    & = & (\pi^*K)\,\lambda_1\wedge\lambda_2.
\end{eqnarray*}
Thus, for a metric of everywhere negative curvature on~$\Sigma_g$, $g\geq 2$,
the connection $1$-form $\alpha$ is a contact form.
In that case, the Reeb vector field $R_{\alpha}$
defines the $S^1$-fibration, so any Morse function on $\Sigma_g$ lifts
to a Bott integral for~$R_{\alpha}$.

However, with
respect to the orientation of $ST^*\Sigma_g$ defined by $\alpha\wedge
\lambda_1\wedge\lambda_2$, the contact structures $\ker\lambda_i$
are positive, whereas $\ker\alpha$ is negative. 
\subsection{Bott-integrable contact structures on $S^1$-bundles}
As an instructive special case of Theorem~\ref{thm:seifert}
(and Theorem~\ref{thm:bott-graph}),
we briefly discuss the construction of $S^1$-invariant
contact structures admitting a Bott-integrable Reeb flows
on arbitrary principal $S^1$-bundles $\pi\co M\rightarrow\Sigma_g$, $g\geq 0$.

Given a connection $1$-form $\alpha$ on $M$, its curvature form
is the $2$-form $\omega$ on $\Sigma_g$ such that $\rmd\alpha=\pi^*\omega$.
The de Rham cohomology class of $\omega$ is related to the Euler class $e$
of the bundle by $e=-[\omega/2\pi]\in H^2(\Sigma_g;\Z)\subset
H^2_{\mathrm{dR}}(\Sigma_g)$; see~\cite[Section 7.2]{geig08}, or
\cite{kela21}, where these concepts are extended to Seifert bundles
over orbifolds. Recall that $e(ST\Sigma_g)=\chi(\Sigma_g)=2-2g$,
and $e(ST^*\Sigma_g)=2g-2$. 

Conversely, given any $2$-form $\omega$ on $\Sigma_g$ with
$e:=-[\omega/2\pi]$ integral, there is a connection $1$-form on the
$S^1$-bundle of Euler class $e$ with curvature form~$\omega$.
It is clear, then, that one can find a connection $1$-form defining
a positive contact structure if and only if $e<0$.
Such contact forms admit a Bott integral as described in the
preceding section.

For $e\geq 0$ (e.g.\ for $ST^*\Sigma_g$ with $g\geq 1$), one needs
a more general construction.

\begin{prop}
Any principal  $S^1$-bundle over $\Sigma_g$ carries an $S^1$-invariant
positive contact form with a Bott-integrable Reeb flow.
\end{prop}

\begin{proof}
We describe the $S^1$-bundle $M$ over $\Sigma_g$ of Euler class $e\in\Z$ as
the gluing
\[ M=\Bigl(\bigl((\Sigma_g\setminus\Int(D^2))\times S^1\bigl)
+(D^2\times S^1)\Bigl)/\!\sim,\]
where $D^2$ is a $2$-disc embedded in $\Sigma_g$,
and $\sim$ denotes the following boundary identification. Write
\[ S^1_0=-\partial(\Sigma_g\setminus\Int(D^2)) \]
for the boundary circle with the \emph{opposite} of the orientation
induced as the boundary of $\Sigma_g\setminus\Int(D^2)$. In the boundary of
\[ M':=(\Sigma_g\setminus\Int(D^2))\times S^1\]
we have the two curves
\[ q:= S^1_0\times\{*\}\;\;\;\text{and}\;\;\; h:=\{*\}\times S^1.\]
On $\partial (D^2\times S^1)$ we have meridian and longitude,
\[ \mu:=\partial D^2\times\{*\}\;\;\;\text{and}\;\;\;
\lambda:=\{*\}\times S^1=1,\]
with $*$ a point on~$\partial D^2$ in the definition of~$\lambda$.
Then, as discussed in \cite{gela18}, for instance, the identification
that gives the desired bundle $M\rightarrow\Sigma_g$ is
\[ \mu=q-eh,\;\;\; \lambda=h.\]
In other words, the gluing may be described by the map
\[ \Phi\co\partial D^2\times S^1\ni(\varphi,\theta)\longmapsto
(\varphi,-e\varphi+\theta)\in S^1_0\times S^1.\]
Notice that the fibre class $h$ becomes identified with
the longitude~$\lambda$, so the $S^1$-fibration of $M'$
and that of $D^2\times S^1$, both given by projection onto
the first factor, define the global $S^1$-fibration of~$M$.

By Lemma~\ref{lem:exact-area}, we have a contact form $\alpha=\rmd\theta+
\lambda$, with $\lambda=\rho(r)\,\rmd\varphi$ on the collar
\[ S^1_0\times[1,1+\varepsilon) \subset (\Sigma_g\setminus\Int(D^2)),\]
where $S^1_0\times{1}\equiv-\partial(\Sigma_g\setminus\Int(D^2))$.
Because of our orientation conventions, the conditions in
Lemma~\ref{lem:exact-area} translate into $\rho'>0$ and $\rho(1)<0$.
The form $\alpha$ pulls back to
\[ \Phi^*\alpha=\rmd\theta+(\rho(r)-e)\,\rmd\varphi. \]
For $e\geq 0$, we have $\rho(1)-e<0$; for $e<0$, we may choose $\rho$ such
that $\rho(1)-e$ is positive.

In either case, we find a contact form $\alpha_M$
on $M$ by extending $\Phi^*\alpha$ as a Lutz form
$h_1(r)\,\rmd\theta+h_2(r)\,\rmd\varphi$
over $D^2\times S^1$. Beware that the ambient orientation is now
given by $\rmd\theta\wedge r\,\rmd r\wedge\rmd\varphi$,
so the \emph{positive} contact condition becomes $\Delta>0$,
in contrast with~\eqref{eqn:Delta}, so the curve $r\mapsto (h_1(r),h_2(r))$
has to wind around the origin in counterclockwise direction.

In the case $e<0$, we may extend $\phi^*\alpha$ in the form
$\rmd\theta+h_2(r)\,\rmd\varphi$, with $h_2(r)=r^2$ near $r=0$.
In the case $e\geq 0$, in order to realise $h_2(r)=\pm r^2$ near
$r=0$, one necessarily has to choose a curve
$r\mapsto (h_1(r),h_2(r))$ with $h_1$ having at least one zero;
with $h_1(r)=-1$ and $h_2(r)=-r^2$ near $r=0$, a single zero
suffices. In other words, in the case $e\geq 0$ the $S^1$-invariant
contact structure $\ker\alpha_M$ necessarily becomes tangent
to the fibres over at least one circle in the base $\Sigma_g$.

Any Morse function on $\Sigma_g$ that on $D^2$ is a function of $r$ only
will lift to a Bott integral
for the Reeb flow of~$\alpha_M$.
\end{proof}

\begin{rem}
It is not difficult to see that any $S^1$-invariant contact
structure on an $S^1$-bundle with $e\geq 0$ necessarily has to
be tangent to the fibres somewhere, see
\cite[Section~1]{lutz77}. Without the $S^1$-invariance, the
necessary and sufficient condition for finding a contact structure
transverse to the fibres weakens to $e\leq 2g-2$ for $g\geq 1$,
as shown by Giroux~\cite{giro01}.
\end{rem}
\subsection{Invariant contact structures on Seifert manifolds}
We now show that any contact structure invariant under a fixed-point
free $S^1$-action admits a Bott-integrable Reeb flow.

\begin{proof}[Proof of Theorem~\ref{thm:seifert}]
Given any contact structure $\xi$ invariant under the $S^1$-action,
by averaging (over~$S^1$) any contact form defining $\xi$, we may
assume that $\xi=\ker\alpha$ with $\alpha$ an $S^1$-invariant contact form.
We want to show that by rescaling $\alpha$ with a suitable
$S^1$-invariant positive function, we can obtain an $S^1$-invariant
contact form with Bott-integrable Reeb flow.

Write $\partial_{\theta}$ for the vector field inducing
the $S^1$-action. Then the function $u:=\alpha(\partial_{\theta})$
on $M$ is $S^1$-invariant. Let $\gamma$ be a connection $1$-form
for the $S^1$-action, that is, $\gamma$ is supposed to be $S^1$-invariant
and $\gamma(\partial_{\theta})=1$. Locally near any fibre such
a connection $1$-form exists, and one can patch them together
using an $S^1$-invariant partition of unity. 

Set $\beta:=\alpha-u\gamma$. This is a so-called \emph{basic}
form (cf.\ \cite{geig22}) for the Seifert fibration, i.e.\
\[ i_{\partial_{\theta}}\beta=0\;\;\;\text{and}\;\;\;
i_{\partial_{\theta}}\rmd\beta=0.\]
For an honest $S^1$-bundle, being basic would mean that
the form is a lift from the base. In the Seifert setting,
for a form to be basic means that it induces a well-defined
form on any (local) surface transverse to the Seifert fibration.
Notice that the $1$-form $\rmd u$ and $2$-form $\rmd\gamma$, too, are basic.

The exterior derivative of a basic $1$-form is basic; so is
the wedge product of basic forms. Also, a basic $3$-form on $M$
is clearly trivial. It follows that
\[ \alpha\wedge\rmd\alpha=\gamma\wedge(u\,\rmd\beta+
\beta\wedge\rmd u).\]
Thus, the contact condition translates into
\[ u\,\rmd\beta+\beta\wedge\rmd u>0,\]
by which we mean that this basic 2-form defines
a positive area form on transversals.

We conclude that $\rmd u\neq 0$ along $\{u=0\}$, the
$S^1$-invariant set where the contact structure is tangent to the
Seifert fibre, which implies that
the zero-level set of $u$ is a compact $2$-dimensional
submanifold of~$M$. Observe the analogy with the work of
Lutz~\cite{lutz77} for $S^1$-bundles and the notion of
`dividing set' on convex surfaces
in the sense of Giroux~\cite{giro91}.

We claim that the surface $\{u=0\}$ does not contain any
singular fibres of the Seifert fibration. Indeed, if the
fibre through some point $p\in M$ were tangent to $\ker\alpha$,
then $\alpha_p$ would induce a well-defined non-trivial
linear form on the quotient vector space $T_pM/\langle
\partial_{\theta}\rangle$. But this linear form would
have to be invariant under the action of the finite
cyclic isotropy group of~$p$, which is impossible.

Next we wish to construct an $S^1$-invariant Morse--Bott function
$f$ on $M$ with the following properties:
\begin{itemize}
\item[-] $f$ coincides with $u$ near $\{u=0\}$;
\item[-] $\sign(f)=\sign(u)$;
\item[-] $\Crit(f)$ consists of finitely many
Seifert fibres, including all singular ones.
\end{itemize}
The base orbifold $\Sigma$ of the Seifert fibration is divided
by the $S^1$-quotient of $\{u=0\}$ into compact surfaces
$\Sigma_{\pm}$ (with boundary) over which $u$ is positive or
negative, respectively. (Of course, the set $\{u=0\}$
may be empty, in which case $\Sigma$ coincides with one
of $\Sigma_{\pm}$.) All the singular points of this
orbifold lie in the interior of these surfaces.
Near the singular fibres, we prescribe $f$ to look like
$\pm c\mp r^2$, with $c$ some positive constant,
in terms of a radial coordinate adapted to
the Seifert fibration, so that the regular fibres
near the singular fibre $\{r=0\}$ foliate concentric
tori $\{r=r_0\}$. The induced function on $\Sigma_{\pm}$
can be extended to a Morse function that coincides
with $u$ near the boundary and takes positive resp.\
negative values on~$\Sigma_{\pm}$. This function, in turn, lifts
to the desired function~$f$.

The function $g:=f/u$ on $M$ is smooth and positive, and we set
$\alpha':=g\alpha$. This is again an $S^1$-invariant contact form
defining~$\xi$, and it satisfies $\alpha'(\partial_{\theta})=f$.
With the Cartan formula for the Lie derivative, the invariance of
$\alpha'$ translates into
\[ i_{\partial_{\theta}}\rmd\alpha'+\rmd f=0;\]
thus, with $R'$ denoting the Reeb vector field of~$\alpha'$,
we have
\[ \rmd f(R')=-\rmd\alpha(\partial_{\theta},R')=0. \]
This concludes the proof that $\xi$ is Bott integrable.
\end{proof}

\begin{rem}
For the construction of the $S^1$-invariant Morse--Bott function one
could also appeal to the results of Wasserman~\cite{wass69} on the denseness
of $G$-invariant Morse functions ($G$ any compact Lie group).
\end{rem}
\section{$S^1$-invariant contact structures on~$S^3$}
\label{section:S1-invariant}
In preparation of the proof of Theorem~\ref{thm:S3T3}
for the $3$-sphere, in this section we take a closer look
at the work of Lutz~\cite{lutz77} on $S^1$-invariant contact
structures on~$S^3$.
\subsection{An invariant trivialisation of $T^*S^3$}
We may think of the $3$-sphere as the unit sphere $S^3\subset\bbH$
in the quaternions. With $I,J,K$ denoting the bundle maps on
$T\bbH$ corresponding to the standard unit quaternions
$\rmi,\rmj,\rmk$, and $r$ the radial coordinate on~$\bbH$,
a frame of $1$-forms on $S^3$ can be defined by
\begin{equation}
\label{eqn:IJK}
\alpha_I:=-r\,\rmd r\circ I,\;\; \alpha_J:=-r\,\rmd r\circ J,
\;\; \alpha_K:=-r\,\rmd r\circ K.
\end{equation}
Under the identification of $\C^2$ with $\bbH$ via
$(z_1,z_2)\mapsto z_2+z_2\rmj$, the first of these $1$-forms
is the standard contact form on $S^3$,
\[ \alpha_I=x_1\,\rmd y_1-y_1\,\rmd x_1+x_2\,\rmd y_2-y_2\,\rmd x_2,\]
whose Reeb flow
\[ t\longmapsto (\rme^{\rmi t}z_1,\rme^{\rmi t}z_2),\;\;\; t\in\R/2\pi\Z,\]
defines the Hopf fibration $\pi\co S^3\rightarrow S^2=\CP^1$,
$(z_1,z_2)\mapsto[z_1:z_2]$.

We want to describe and classify (up to homotopy of nowhere zero
$1$-forms) the contact structures invariant
under this $S^1$-action.

The contact form $\alpha:=\alpha_I$ is clearly invariant, in fact it
is the connection $1$-form on the Hopf fibration.
The other two forms in \eqref{eqn:IJK}, however, are
the Liouville--Cartan forms of this bundle, and not $S^1$-invariant.

Instead, we define an $S^1$-invariant trivialisation of the
cotangent bundle $T^*S^3$ as follows. Regard the base $S^2$ of 
the Hopf fibration as the unit sphere in $\R^3$ with cartesian
coordinates $x_1,x_2,x_3$. Then the $1$-forms
\[ \alpha_i^0:=x_i\alpha+\rmd x_i,\;\;\; i=1,2,3,\]
on $S^3$, where by abuse of notation we identify the function $x_i\circ\pi$
with~$x_i$, are pointwise linearly independent. More generally, one may identify
$S^2$ with any embedded $2$-sphere $S\subset\R^3$, and then
the $1$-forms
\[ \alpha_i:=\psi_i\alpha+\rmd x_i,\;\;\;i=1,2,3,\]
with $\psi_i\co S\rightarrow\R$ (or the corresponding
$S^1$-invariant function on~$S^3$), define a frame provided
\[ \psi_1\partial_{x_1}+\psi_2\partial_{x_2}+\psi_3\partial_{x_3}
\;\;\;\text{is transverse to~$S$;}\]
this follows from the computation
\[ \alpha_1\wedge\alpha_2\wedge\alpha_3=\alpha\wedge
i_{\sum\psi_j\partial_{x_j}}(\rmd x_1\wedge\rmd x_2\wedge\rmd x_3).\]
For instance, for the proof of Proposition~\ref{prop:lutz}
below, Lutz flattens the $2$-sphere near the north pole $(1,0,0)$
and works with a frame $(\alpha_1,\alpha_2,\alpha_3)$
homotopic to $(\alpha_1^0,\alpha_2^0,\alpha_3^0)$
that near the north pole is of the form
$(\alpha,\rmd x_2,\rmd x_3)$, which simplifies homotopical
calculations considerably.
\subsection{The homotopy classification of invariant $1$-forms}
Given a nowhere zero $1$-form $\sigma$ on $S^3$, not
necessarily $S^1$-invariant, we can write it as
\[ \sigma=\mu_1\alpha_1+\mu_2\alpha_2+\mu_3\alpha_3.\]
By rescaling $\sigma$, we may assume that
\[ \mu_1^2+\mu_2^2+\mu_3^2=1,\]
so that $\sigma$ is described by the map
\[ \mu:=(\mu_1,\mu_2,\mu_3)\co S^3\longrightarrow S^2. \]
The Hopf invariant $H(\sigma):=H(\mu)\in\Z$ determines $\sigma$ up to homotopy
through nowhere zero $1$-forms.

Now suppose $\sigma$ is $S^1$-invariant; equivalently, the
$\mu_i$ are $S^1$-invariant. Then $\mu$ factors through the
Hopf fibration~$\pi$, that is, $\mu=\omu\circ\pi$ for some
$\omu\co S^2\rightarrow S^2$. It follows that
\begin{equation}
\label{eqn:Hmu}
H(\mu)=(\deg\omu)^2H(\pi)=(\deg\omu)^2;
\end{equation}
this formula is a straightforward consequence of the
differential forms definition of the Hopf invariant~\cite{botu82}.
In particular, not all homotopy classes of nowhere zero $1$-forms
are realised by $S^1$-invariant $1$-forms.
However, the restriction on $S^1$-invariant forms coming from
\eqref{eqn:Hmu} is the only one, and actually one can find an
$S^1$-invariant \emph{contact} form in each allowable homotopy class.

\begin{prop}[Lutz]
\label{prop:lutz}
If $\sigma$ is an $S^1$-invariant and nowhere zero $1$-form on $S^3$,
its Hopf invariant $H(\sigma)$ is a square. Conversely,
for every $n\in\N_0=\{0,1,2,\ldots\}$ there is an $S^1$-invariant contact
form $\sigma_n$ on $S^3$ with $H(\sigma_n)=n^2$.\qed
\end{prop}

\begin{rem}
This has been proved by Lutz~\cite[Section~3]{lutz77}.
The statements in Sections 4.3 and 4.4 of that paper seem to suggest that
\emph{every} homotopy class of tangent $2$-plane fields can be realised by an
$S^1$-invariant $1$-form (or even a contact form), but this is clearly
in error. Lutz also discusses the classification up to
equivariant diffeomorphism.
\end{rem}

\begin{ex}
\label{ex:H-frame}
With the frame $\alpha_i=x_i\alpha+\rmd x_i$ and $S=S^2$, we have
\[ \sum_{i}x_i\alpha_i=\sum_ix_i^2\alpha+\sum_ix_i\,\rmd x_i=
\alpha\;\;\;\text{on $TS^3$},\]
which means that $\omu_{\alpha}=\mathrm{id}_S^2$. It follows that
$H(\alpha)=1$.
\end{ex}
\subsection{Remarks on the Hopf invariant}
The Hopf invariant of a map $S^3\rightarrow S^2$ is a well-defined integer.
The Hopf invariant $H(\sigma)$ of a nowhere vanishing $1$-form
$\sigma$ on $S^3$, however, depends on the choice of trivialisation
$(\alpha_1,\alpha_2,\alpha_3)$ of $T^*S^3$ with respect to which we
identify $\sigma$ with a map $\mu\co S^3\rightarrow S^2$.
We briefly discuss how $H(\sigma)$ transforms under a change of
trivialisation, and how our convention for the Hopf invariant relates to the
equally common choice of trivialisation given by
$(\alpha_I,\alpha_J,\alpha_K)$.

Write $\sigma=\mu_1\alpha_1+\mu_2\alpha_2+\mu_3\alpha_3$.
We take the orientation and bundle metric on $T^*S^3$ that make $(\alpha_1,
\alpha_2,\alpha_3)$ a positive orthonormal frame. With
respect to a second positive orthonormal frame, $\sigma$ is described
by a map $\mu'\co S^3\rightarrow S^2$ of the form
\[ \mu'(p)=A(p)
\begin{pmatrix}
\mu_1(p)\\
\mu_2(p)\\
\mu_3(p)
\end{pmatrix}\]
for some $A\co S^3\rightarrow\mathrm{SO}(3)$ describing the change of frame.

Decompose $S^3$ into two closed hemispheres as $S^3=D^3_+\cup_{S^2}D^3_-$.
Up to homotopy, we may assume that
\[ \mu=\begin{pmatrix}1\\0\\0\end{pmatrix}\;\;\text{on $D^3_-$}\;\;\;
\text{and}\;\;\;A=\mathrm{id}\;\;\text{on $D^3_+$}.\]
Then
\[ A\mu=\begin{cases}
\mu                          & \text{on $D^3_+$},\\
(a_{11},a_{21},a_{31})^{\ttt} & \text{on $D^3_-$},
\end{cases}\]
where $(a_{11},a_{21},a_{31})^{\ttt}$ denotes the first column vector of~$A$.
Hence
\[ [\mu']=[\mu]+[(a_{11},a_{21},a_{31})^{\ttt}]\;\;\;\text{in}
\;\; \pi_3(S^2)\cong\Z.\]
We conclude that
\[ H(\mu')=H(\mu)+c_A,\]
with an integer  $c_A$ depending on $A$ only.

With respect to the frame $\alpha_I,\alpha_J,\alpha_K$, the contact
form corresponds to the constant map, with Hopf invariant~$0$.
By comparing this with
Example~\ref{ex:H-frame}, we arrive at the following statement.

\begin{lem}
\label{lem:HH'}
The Hopf invariant $H$ of nowhere vanishing $1$-forms computed with respect
to the frame $(\alpha_1,\alpha_2,\alpha_3)$ and the invariant $H'$
computed with respect to $(\alpha_I,\alpha_J,\alpha_K)$
are related by $H=H'+1$.\qed
\end{lem}

An invariant definition of the Hopf invariant for tangent $2$-plane
fields on~$S^3$ (i.e.\ the nowhere
vanishing $1$-form defining such a plane field),
independent of a choice of a trivialisation of $T^*S^3$,
is the $d_3$-invariant of Gompf~\cite{gomp98}. This takes
values in $\Z+\frac{1}{2}$, and for the standard contact structure
$\xist$ it takes the value $d_3(\xist)=-\frac{1}{2}$.
Gompf's invariant is related to $H'$ via
\begin{equation}
\label{eqn:d3H'}
d_3=-H'-\frac{1}{2};
\end{equation}
see~\cite[p.~114]{elfr09}. This formula can also be derived
from the considerations in \cite[Section~4.3]{geig08}.

The next lemma addresses the behaviour of the Hopf invariant
$H'$ under contact connected sums. Given two connected contact manifolds
(of the same but arbitrary odd dimension), their connected sum carries
a well-defined contact structure. We shall return to this issue in 
Section~\ref{section:cont-connect}.

\begin{lem}
\label{lem:connectH'}
Under the contact connected sum of contact structures
on $S^3$, the Hopf invariant $H'$ behaves additively, that is,
for contact structures $\xi,\xi'$ on $S^3$ we have
\[ H'(\xi\#\xi')=H'(\xi)+H'(\xi').\]
\end{lem}

\begin{proof}
For the $d_3$-invariant, the connected sum formula
\[ d_3(\xi\#\xi')=d_3(\xi)+d_3(\xi')+\frac{1}{2}\]
has been established in \cite[Lemma~4.2]{dgs04}. The formula in
the lemma then follows from~\eqref{eqn:d3H'}.
\end{proof}

Finally, we can make a sanity check concerning the relation between
$H$ and $H'$ by using the following observation of Giroux,
cf.~\cite[p.~115]{elfr09}: the Lutz twist on $\xist$ along $k$ fibres
of the Hopf fibration produces an $S^1$-invariant contact structure
$\xi_k$ with $H'(\xi_k)=k(k-2)$. This is consistent
with Proposition~\ref{prop:lutz} and Lemma~\ref{lem:HH'}, since
\[ H(\xi_k)=H'(\xi_k)+1=(k-1)^2.\]
\subsection{Bott integrals}
The existence of a Bott integral for the $S^1$-invariant contact forms
$\sigma_n$ in Proposition~\ref{prop:lutz} follows from
Theorem~\ref{thm:seifert}. Alternatively, one can use the
explicit description of these $S^1$-invariant forms in \cite{lutz77}
or Giroux's observation we just mentioned
to construct a Morse function on the base that lifts to a function
invariant under the Reeb flow of~$\sigma_n$.

The $\sigma_n$ are obtained from the connection $1$-form $\alpha$ by
a modification over a collection of circles in the base $S^2$.
Outside this modification, the Reeb flow defines the Hopf fibration,
and any Morse function on the base will lift to a Morse--Bott function.
Over an annulus around each circle in the collection,
the connection $1$-form is replaced by a Lutz form $h_1(r)\,\alpha
+h_2(r)\,\rmd\varphi$, where $\varphi$ is the angular coordinate
on the annulus in circle direction, and $r$ is a transverse
coordinate. Thus, any Morse function on the base that is a function
of $r$ only in each of these annular neighbourhoods will lift
to the desired Bott integral.
\section{Integrable Reeb flows on contact connected sums}
\label{section:cont-connect}
In this section we adapt the contact connected sum construction
to the case of Bott-integrable contact forms. This will allow
us to prove Theorem~\ref{thm:S3T3} for the $3$-sphere.
\subsection{The model handle}
We follow Weinstein's description~\cite{wein91} of
contact surgery for the model of a contact connected
sum (or a symplectic $1$-handle); see \cite[Chapter~6]{geig08}
for further context. Thus, on $\R^4$ with
cartesian coordinates $x,y,z,t$ and standard symplectic form
\[ \omega = \rmd x\wedge\rmd y+\rmd z\wedge\rmd t,\]
we consider the Liouville vector field
\[ Y=\frac{1}{2}x\partial_x+\frac{1}{2}y\partial_y
+2z\partial_z-t\partial_t.\]
The $1$-form $i_Y\omega$ induces a contact form on any
hypersurface transverse to~$Y$.

The hypersurface we wish to consider is constructed
as follows. We consider a smooth function $h=h(\rho,t)$
on $\R^+\times\R$ with a regular level set $\{h=0\}$ 
as shown in Figure~\ref{figure:model-handle}.
We require that the vector field $\rho\partial_{\rho}
-t\partial_t$ be positively transverse to this level set,
i.e.\
\[ \rho h_{\rho}-th_t>0\;\;\;\text{along $\{h=0\}$}.\]
Set $H(x,y,z,t):=h(\rho,t)$ with $\rho:=x^2+y^2+z^2$.
Then the $3$-dimensional hypersurface $M:=\{H=0\}$ in $\R^4$
can be visualised (up to one missing dimension)
as being obtained by rotating the curve $\{h=0\}$
about the $t$-axis. The Liouville vector field
$Y$ is clearly transverse to~$M$. One may in fact
compute explicitly that $\rmd H(Y)>0$.

\begin{figure}[h]
\labellist
\small\hair 2pt
\pinlabel $\rho$ [t] at 211 215
\pinlabel $t$ [r] at 0 425
\pinlabel $\{t=1\}$ [b] at 60 397
\pinlabel $\{t=-1\}$ [t] at 60 37
\pinlabel $\rho_0$ [tr] at 51 216
\pinlabel $h>0$ [l] at 113 280
\pinlabel $h<0$ [l] at 22 346
\pinlabel $\{h=0\}$ [bl] at 90 97
\endlabellist
\centering
\includegraphics[scale=0.45]{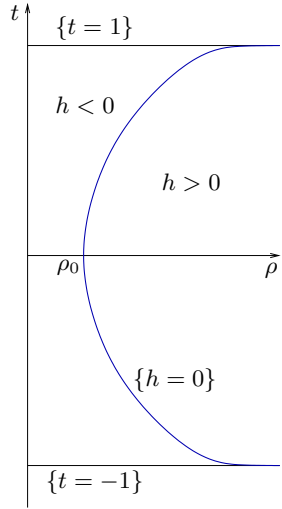}
  \caption{The model for the connected sum}
  \label{figure:model-handle}
\end{figure}

The Hamiltonian vector field $X_H$, defined by $-\rmd H=\omega
(X_H,\,.\,)$, takes the form
\[ X_H=-2yh_{\rho}\partial_x+2xh_{\rho}\partial_y
-h_t\partial_z+2zh_{\rho}\partial_t.\]
The function $F\co (x,y,z,t)\mapsto x^2+y^2$ on
$\R^4$ then satisfies $\rmd F(X_H)=0$. The restriction
$f$ of $F$ to $M$ likewise is an integral of~$X_H$.

As is well known, cf.\ \cite[Lemma~1.4.10]{geig08},
the Reeb vector field $R$ of the contact form
$\alpha:=(i_Y\omega)|_{TM}$ is a rescaling
of $X_H$; in fact,
\[ R=\frac{X_H}{\rmd H(Y)}.\]
It follows that the function $f$ is an integral of~$R$.

Write $(\rho_0,0)$ for the point where the curve
$\{h=0\}$ intersects the $\rho$-axis. We may choose
$h$ such that $h_{tt}(\rho_0,0)<0$, ensuring
convexity of the curve at that point. Under this
assumption, we have found the desired Bott integral.

\begin{lem}
If $h_{tt}(\rho_0,0)<0$, the
function $f\co M\rightarrow\R$ is Morse--Bott.
\end{lem}

\begin{proof}
We are going to show that $\Crit(f)$ consists of
two Reeb orbits passing over the handle
through the points $(x,y,z)=(0,0,\pm\sqrt{\rho_0})$
on the belt sphere $M\cap\{t=0\}$, and a periodic
orbit making up the intersection of the belt sphere
with the $xy$-plane.

First we look at the set
\[ \{p\in\R^4\co H(p)=0,\; \rmd_pF=0\}=
\{ H=0, x=y=0\}.\]
This consists of the two curves in the $zt$-plane
described by $h(z^2,t)=0$. Since the restriction of
$X_H$ to the $zt$-plane is tangent to that plane,
these two curves are Reeb orbits, as we should expect.
Along those orbits, $x$ and $y$ constitute transverse
coordinates, so $f=x^2+y^2$ has a non-degenerate
minimum along these components of $\Crit(f)$.

The remaining part of $\Crit(f)$ is made up
of points $p\in M$ where $\rmd_pF$ is non-zero,
but proportional to $\rmd_pH$. Comparing those two
differentials, we see that this condition translates into
\[ x^2+y^2>0,\;\;\;zh_{\rho}=0,\;\;\;\text{and}\;\;\;h_t=0.\]
The third condition is satisfied only at $t=0$;
with $h_{\rho}(\rho_0,0)>0$ the second condition
is then equivalent to $z=0$.
The first condition is then automatic, since $\rho_0>0$.
This describes the
intersection of the belt sphere with the $xy$-plane,
as claimed, and again we see that this is indeed an
orbit of~$H$.

For $|t|$ small we have $h_{\rho}(\rho,t)\neq 0$,
so the implicit function theorem gives us a smooth function
$t\mapsto\rho(t)$ such that the set $\{h=0\}$
is described by $h(\rho(t),t)=0$.
Therefore, for $|t|$ small, we may regard $f$ as a function
of the variables $z$ and $t$:
\[ x^2+y^2=\rho(t)-z^2=:f(z,t).\]
Then
\[ h_{\rho}(\rho(t),t)\rho'(t)+h_t(\rho(t),t)=0,\]
so from $h(\rho_0,0)=0$ and $h_t(\rho_0,0)=0$ we conclude that
$\rho(0)=\rho_0$ and $\rho'(0)=0$. By differentiating
the implicit equation a second time and evaluating
at $(\rho_0,0)$, we get
\[ h_{\rho}(\rho_0,0)\rho''(0)+h_{tt}(\rho_0,0)=0.\]

The function $f(z,t)$ is critical at $z=t=0$, as it should be,
and the components of its Hessian are
\[ f_{zz}(0,0)=-2,\;\;
f_{zt}(0,0)=f_{tz}(0,0)=0,\;\;
\text{and}\;\;
f_{tt}(0,0)=\rho''(0)=-\frac{h_{tt}(\rho_0,0)}{h_{\rho}(\rho_0,0)}>0;\]
where we have used the convexity assumption. This non-degenerate
and indefinite Hessian tells us that the critical orbit
contained in the belt sphere is of hyperbolic type.
\end{proof}
\subsection{The contact connected sum}
\label{subsection:cont-connect}
We now use the model handle to carry out the connected sum
of Bott-integrable contact manifolds. Recall that thanks to the
contact disc theorem \cite[Theorem~2.6.7]{geig08}
there is a well-defined connected sum
operation for connected contact manifolds $(M_{\pm},\xi_{\pm})$,
i.e.\ the operation we are about to describe leads to a contact
structure $\xi_-\#\xi_+$ on $M_-\#M_+$ that is unique up
to diffeomorphism.

\begin{thm}
\label{thm:cont-connect}
Let $(M_{\pm},\alpha_{\pm})$ be two connected contact $3$-manifolds
with Bott-integrable Reeb flows. Then the contact connected sum
\[ (M_-\#M_+,\ker\alpha_-\#\ker\alpha_+)\]
admits a Bott-integrable Reeb flow.
\end{thm}

\begin{proof}
Consider the model handle from the preceding section.
On the affine hyperplanes $\{t=\pm 1\}$, the $1$-form $i_Y\omega$
induces the contact forms
\[ \alpha^0_{\pm}:=(i_Y\omega)|_{T\{t=\pm 1\}}=\pm\rmd z+\frac{1}{2}
(x\,\rmd y-y\rmd x)\]
with Reeb vector field $R^0_{\pm}=\pm\partial_z$. The function
$f(x,y,z)=x^2+y^2$ is a Bott integral for $R^0_{\pm}$.

Now, by Proposition~\ref{prop:create-orbit} we may assume
that $(M_{\pm},\alpha_{\pm})$ contains
an isolated elliptic Reeb orbit~$\gamma_{\pm}$. By Theorem~\ref{thm:nbhd},
these orbits have neighbourhoods where the contact form and the Bott integral
look just like the neighbourhoods of the curve
\[ \theta\longmapsto (0,0,\pm\theta) \in \{t=\pm1\}/z\sim z+2\pi.\]
Therefore, by choosing the model handle small enough, we can ensure
that the surgery in the model is effected inside 
a neighbourhood of $(0,0,0,\pm 1)\in\{t=\pm 1\}$ that can be identified
with a neighbourhood of points $p_{\pm}\in\gamma_{\pm}$ in
$(M_{\pm},\alpha_{\pm})$ --- where by `identify' we mean that we have a
diffeomorphism that respects both the contact form and the Bott integral.
This allows us to carry out the connected sum of $(M_{\pm},\alpha_{\pm})$
in these neighbourhoods of $p_{\pm}$
in such a way that we obtain a Bott-integrable contact manifold.
In $M_-\#M_+$, the orbits $\gamma_{\pm}$ induce a single critical
Reeb orbit, and we have created a new critical Reeb orbit in the belt sphere
of the $1$-handle.
\end{proof}
\subsection{Integrable Reeb flows on the $3$-sphere}
\label{subsection:S3}
We are now ready to prove Theorem~\ref{thm:S3T3} for the
$3$-sphere. In Section~\ref{section:S1-invariant} we found
Bott-integrable $S^1$-invariant contact structures on $S^3$ realising
the values $k(k-2)$, $k\in\N_0$ of the Hopf invariant~$H'$.
In particular, we have the values $H'=-1$ and $H'=3$ for $k=1$ and $k=3$,
respectively. Thanks to Lemma~\ref{lem:connectH'} and
Theorem~\ref{thm:cont-connect}, by taking connected sums we can realise
any integer as $H'$ of a Bott-integrable contact structure.

Now, on $S^3$ there is a unique tight contact structure,
namely $\xist=\ker\alpha$ (with $H'(\xist)=0$),
and a unique overtwisted contact structure for every value of~$H'$;
see~\cite{elia89,elia92}. So it only remains to ensure that
we also have an \emph{overtwisted} integrable contact structure with $H'=0$.
The easiest way to obtain such a structure is to perform
a \emph{full} Lutz twist on $\xist$ along a Hopf fibre; this produces an
overtwisted (and integrable) contact structure homotopic to $\xist$
as a tangent $2$-plane field~\cite[Lemma~4.5.3]{geig08}.
\section{Integrable Reeb flows constructed via open books}
\label{section:open-book}
In this section we describe an open book decomposition
of the trivial circle bundle over a closed, oriented surface.
The contact structure adapted to this open book (in the sense
of Giroux) is then seen to be Bott integrable. This leads
to a construction of Bott-integrable contact structures
on Seifert manifolds that gives control over the homotopy
type of the contact structure.

For other interesting aspects of the connection between Reeb dynamics
and open books, see~\cite{kkvk}.
\subsection{An open book decomposition of $\Sigma_g\times S^1$}
Write $\Sigma_g$ for the closed, oriented surface of genus $g\geq 0$.
We wish to construct an open book decomposition of $\Sigma_g\times S^1$.
To this end, start with a section $\Sigma_g\equiv\Sigma_g\times\{*\}$
of the flow in $S^1$-direction. We will modify this into an honest
surface of section, with one positive and one negative boundary
component. The $S^1$-translates of this surface of section
then define the pages of the open book.
For the background on open books, see \cite[Sections 4.4.2 and 7.3]{geig08}.

The following construction of the desired surface of section
is taken from~\cite{agz23}, to which we refer for pictures.
An alternative description in terms of
a cancelling pair of surgeries can be found in~\cite{etoz06}.
Let $D^2_{\pm}$ be two disjoint discs in $\Sigma_g$.
Remove the two cylinders $C_{\pm}$ over $\Int(D^2_{\pm})$ from
$\Sigma_g\times S^1$. Pick points $p_{\pm}\in\partial D^2_{\pm}$
and a simple path $\gamma$ in
\[ \Sigma_g':=\Sigma_g\setminus\bigl(\Int(D^2_-)\cup\Int(D^2_+)\bigr) \]
joining $p_-$ and~$p_+$. We write $\lambda_{\pm}:=\{p_{\pm}\}\times S^1$
for longitudes on $C_{\pm}$, and $\mu_{\pm}:=\partial D^2_{\pm}$
for meridians. Notice that the meridians $\mu_{\pm}$ carry the orientation
as boundary curves of $D^2_{\pm}$ rather than as boundary
of~$\Sigma_g'$.

Now desingularise the union of $\Sigma_g'$
and the vertical annulus $A:=\gamma\times S^1$ ---
these two surface intersecting each other in~$\gamma$ --- into
a surface $\Sigma_g''$ with the $S^1$-fibre positively transverse to it,
and with boundary the two curves
\[ -\mu_--\lambda_-\;\;\;\text{and}\;\;\;-\mu_++\lambda_+. \]
Write $0_{\pm}$ for the centres of $D^2_{\pm}$. We can now 
find helicoidal annuli $A_{\pm}$, positively transverse to the $S^1$-fibres,
and with oriented boundaries
\[ \bigl(\{0_-\}\times (-S^1)\bigr)\cup (\mu_-+\lambda_-)
\;\;\;\text{and}\;\;\;
\bigl(\{0_+\}\times S^1\bigr)\cup (\mu_+-\lambda_-).\]
Then $\Sigma_g^0:=\Sigma_g''\cup A_-\cup A_+$ is a surface of section with one
negative and one positive $S^1$-fibre as boundary curves.

We now want to convince ourselves that $\Sigma_g^0$, as page of the open book,
gives rise to an open book decomposition of $\Sigma_g\times S^1$
with binding $\{0_{\pm}\}\times(\pm S^1)$,
and monodromy a left- and right-handed Dehn twist, respectively, along
a boundary parallel curve. For additional details (and pictures) the
reader may wish to consult~\cite{etoz06}.

The pages are simply the $S^1$-translates of $\Sigma_g^0$.
To find the monodromy, one needs an $S^1$-invariant vector field
positively transverse to the pages, and pointing in meridional direction
$\pm\mu_{\pm}$ near the binding. Such a vector field can be defined
by taking the vector field $\partial_{\theta}$ in
$S^1$-direction outside the cylinders $C_{\pm}$, and then extending it
into the cylinders as a vector field making a $\pi/2$ turn along radial
lines from $\partial_{\theta}$ to $\pm\partial_{\varphi}$. One
can then read off the claimed monodromy as the return map
of this vector field.
\subsection{Integrable Reeb flows on $\Sigma_g\times S^1$}
We now use the open book description of $\Sigma_g\times S^1$
to find a contact form with Bott-integrable Reeb flow.
The explicit construction gives us control over the Euler
class of the contact structure.
\subsubsection{An exact area form adapted to the monodromy}
\label{subsubsection:area-monodromy}
The page $\Sigma:=\Sigma_g^0$ of the open book is a copy of $\Sigma_g$ with
two open discs removed. We parametrise collars of the two boundary
components $\partial_{\pm}\Sigma$ as $(-1,0]\times\partial_{\pm}\Sigma$,
with coordinates $(r,\varphi)\in (-1,0]\times\R/2\pi\Z$.
We can describe boundary parallel Dehn twist supported inside these collar
neighbourhoods by
\[ \psi_{\pm}\co(r,\varphi)\longmapsto\bigl(r,\varphi\pm\chi(r)\bigr),\]
where $\chi\co (-1,0]\rightarrow [0,2\pi]$ is a smooth monotone
decreasing function with $\chi\equiv 2\pi$ near $r=-1$, and $\chi\equiv 0$
near $r=0$. With these choices, $\psi_+$ describes a right-handed
Dehn twist; $\psi_-$ is left-handed. We write $\psi$ for the
extension of $\psi_{\pm}$ over $\Sigma$ as the identity map
outside the collars.

By Lemma~\ref{lem:exact-area}, we find an exact area form $\rmd\lambda$
on $\Sigma$, where $\lambda$ looks like $\rho(r)\,\rmd\varphi$
with $\rho'>0$ on the collars.

\begin{lem}
\label{lem:exact-mono}
The monodromy $\psi$ is an exact symplectomorphism for $\rmd\lambda$,
that is, $\psi^*\lambda-\lambda=\rmd\tau$ for some smooth function
$\tau\co\Sigma\rightarrow\R^+$.
\end{lem}

\begin{proof}
In the collar neighbourhoods we compute
\[ \psi_{\pm}^*(\rho(r)\,\rmd\varphi)=\rho(r)\,\rmd\varphi\pm\rho(r)\chi'(r)
\,\rmd r =\rho(r)\,\rmd\varphi\pm\rmd\int_{-1}^r \rho(s)\chi'(s)\,\rmd s.\]
We then set $\tau$ equal to a sufficiently large positive constant $\tau_0$
outside the collars such that the extension over the collars defined by
\[ \tau(r,\varphi):=\tau_0\pm\int_{-1}^r \rho(s)\chi'(s)\,\rmd s\]
is positive. Notice that $\tau$ is locally constant
near $\partial_{\pm}\Sigma$.
\end{proof}
\subsubsection{$\Sigma_g\times S^1$ as an open book}
The pair $(\Sigma,\psi)$ consisting of a surface with boundary
and a diffeomorphism of $\Sigma$ equal to the identity near $\partial\Sigma$
gives rise to a mapping torus
\[ \Sigma(\psi):=[0,2\pi]\times\Sigma/(2\pi,x)\sim(0,\psi(x)),\]
with boundary
\[ \partial\Sigma(\psi)=S^1\times\partial\Sigma.\]
The open book determined by $(\Sigma,\psi)$ is then the closed $3$-manifold
\[ M(\psi):=\Sigma(\psi)\cup_{S^1\times\partial\Sigma}
(D^2\times\partial\Sigma).\]
In this way, for $\Sigma=\Sigma_g^0$ and $\psi$ the diffeomorphism
made up of the two boundary parallel Dehn twists, we recover
$\Sigma_g\times S^1$.
\subsubsection{The contact structure adapted to the open book}
There is a construction of contact structures adapted to an open book
due to Thurston and Winkelnkemper~\cite{thwi75} in dimension~$3$,
and generalised to higher dimensions by Giroux~\cite{giro02}.
For the construction of a contact form with Bott-integrable Reeb flow,
even though we are in dimension~$3$, we need to rely on the latter.
Everything required here can be found in Sections 4.4.2 and 7.3
of~\cite{geig08}.

For this construction, we replace the mapping torus $\Sigma(\psi)$
by a diffeomorphic copy better suited to the definition of
a contact form. On $\R\times\Sigma$ we have a free $\Z$-action 
generated by
\[ (t,x)\longmapsto\bigl(t-\tau(x),\psi(x)\bigr).\]
The condition $\tau>0$ guarantees that a slice $\{0\}\times\Sigma$
is sent to a disjoint copy of it. The quotient $(\R\times\Sigma)/\Z$
is then diffeomorphic to $\Sigma(\psi)$. We continue to write
$\Sigma(\psi)$ for this new model.
The function $\tau$ is locally constant near $\partial_{\pm}\Sigma$,
and $\Sigma_g\times S^1=M(\psi)$ is then
obtained by filling in $D^2\times\partial\Sigma$ as before.

Thanks to Lemma~\ref{lem:exact-mono}, the contact form $\rmd t+\lambda$
on $\R\times\Sigma$, with Reeb vector field~$\partial_t$,
is invariant under the $\Z$-action and hence descends
to $\Sigma(\psi)$. Near the boundary, it is a Lutz form,
and the extension over $D^2\times\partial\Sigma$,
as described in \cite[Section~4.4.2]{geig08},
is also a suitable Lutz form $h_1(s)\, \rmd\varphi+h_2(s)\,\rmd\theta$,
where $(s,\theta)$ are polar coordinates on~$D^2$, and $\varphi$
the angular coordinate along $\partial_{\pm}\Sigma$ as before.
The boundary conditions on $h_1(s),h_2(s)$ at $s=1$ are such that
the contact forms on $\Sigma(\psi)$ and $D^2\times\partial\Sigma$
glue smoothly; near $s=0$ the form looks like $2\,\rmd\varphi+s^2\rmd\theta$,
so that it is smooth at $s=0$. In particular, the binding
$\{0\}\times\partial\Sigma$ of the open book consists of Reeb orbits.
We write $\alpha_g$ for the contact form on $\Sigma_g\times S^1$ thus obtained.

The contact structure $\xi_g=\ker\alpha_g$ on $\Sigma_g\times S^1$ is
`supported' by the open book, in the sense that $\rmd\alpha_g$ defines
a positive area form on each page, and $\alpha_g$ evaluates positively
on the binding components, oriented as the boundary of a page.
In dimension~$3$, these conditions define a unique contact structure up to
isotopy.
\subsubsection{The Bott integral}
\label{subsubsection:SigmagS1-bott}
The $\R$-invariant extension of any
Morse function on $\Sigma$ that depends on $r$ only
inside the collar neighbourhoods of $\partial_{\pm}\Sigma$ will
descend to a Bott integral with isolated critical orbits
on $\Sigma(\psi)$, since the $r$-coordinate
is preserved by the Dehn twists. This Morse function may be chosen
to be strictly increasing towards the boundary, and one can then extend it
over $D^2\times\partial\Sigma$ as a function of the radial
coordinate $s$ on the $D^2$-factor, with an isolated non-degenerate
maximum at the centre. This turns the two binding components
into isolated Reeb orbits in the critical set of the Bott integral.
\subsubsection{The Euler class}
\label{subsubsection:euler}
In order to determine which homotopy class of tangent $2$-plane fields is
realised by the contact structure $\xi_g$, we first of all need to
compute its Euler class.
We write $\PD[S^1]\in H^2(\Sigma_g\times S^1;\Z)$ for the
Poincar\'e dual of the fibre class~$[S^1]$.

\begin{lem}
\label{lem:euler-xig}
The Euler class of $\xi_g$ is $e(\xi_g)=-2g\,\PD[S^1]$.
\end{lem}

\begin{proof}
The Euler class $e(\xi_g)$ is Poincar\'e dual to the transverse
self-intersection of $M_g:=\Sigma_g\times S^1$ in the total space
of the $\R^2$-bundle $\xi_g$ over~$M_g$.
Notice that a transverse and isotopic copy $M_g'$ of $M_g$ inherits an
orientation from~$M_g$, and the $0$-section $M_g$ is cooriented by the
orientation of~$\xi_g$ (given by $\rmd\alpha_g|_{\xi_g}$).
This defines an orientation on the $1$-dimensional
submanifold $M_g\cap M_g'\in M_g$, so this submanifold represents
a well-defined element of $H_1(M_g;\Z)$.

Any Bott integral $f$ for $\alpha_g$ as in
Section~\ref{subsubsection:SigmagS1-bott} has isolated critical
Reeb orbits corresponding to the critical points of the chosen Morse
function on $\Sigma=\Sigma_g^0$, plus the two binding orbits
$\partial_{\pm}\Sigma$. Notice that $[\partial_{\pm}\Sigma]=\pm[S^1]$.

A section $Y$ of $\xi_g$ with zeros along the critical
Reeb orbits is defined by~\eqref{eqn:Y}. On a local surface
of section to the Reeb flow near such a critical orbit, $\rmd\alpha_p$
defines a positive area form. By writing $f$ in normal
form near such a critical orbit as $\pm(x^2+y^2)$ or $xy$ in terms
of transverse cartesian coordinates~$x,y$, one sees that $Y$ defines a
section of $\xi_g$ transverse to the zero section, which implies that
$e=\PD[\{Y=0\}]$. Also, one sees as in the $2$-dimensional Poincar\'e--Hopf
theorem that elliptic orbits in $\{Y=0\}$ (corresponding to an index $+1$
singularity on a local surface of section) carry the orientation
defined by the Reeb flow, and hence define the class~$[S^1]$,
whereas hyperbolic ones (index $-1$) represent~$-[S^1]$.

The Morse function on $\Sigma$ used to construct $f$ extends to a
Morse function on $\Sigma_g$ with two additional elliptic points.
So the indices of the critical points on $\Sigma$ add up to
$\chi(\Sigma_g)-2=-2g$. The two binding orbits give
no further contribution, since $[\partial_{\pm}\Sigma]=\pm[S^1]$;
in fact, the pair is homologous to zero as boundary of~$\Sigma$.
\end{proof}
\subsection{Integrable Reeb flows on Seifert manifolds}
\label{subsection:seifert-via-surgery}
Seifert manifolds are the constituents, and hence special cases
of graph manifolds. So the existence of \emph{some} Bott-integrable
Reeb flow on a given Seifert manifold is
a corollary of Theorem~\ref{thm:bott-graph}. If one tries to
prove the analogue of Theorem~\ref{thm:S3T3} for Seifert
manifolds, however, one needs better control over the
homotopy classes of contact structures that admit Bott-integrable
Reeb flows. For this, the construction of such contact structures
starting from the open book for $\Sigma_g\times S^1$ should prove
useful, thanks to our computation of $e(\xi_g)$ in
Lemma~\ref{lem:euler-xig}.

Seifert fibred manifolds with oriented fibres and base orbifold are
obtained from $\Sigma_g\times S^1$ via Dehn surgeries
along a finite number of $S^1$-fibres; see~\cite{gela18}.
This means that we choose disjoint $2$-discs $D^2_i$
in~$\Sigma_g$, remove $\Int(D^2_i)\times S^1$ from
$\Sigma_g\times S^1$, and then reglue solid tori.

The $D^2_i$ may be chosen in a region of $\Sigma_g^0\subset\Sigma_g$
outside the support of the monodromy $\psi$ of the open book.
The exact area form chosen in
Section~\ref{subsubsection:area-monodromy} may be assumed
to look like $r^2\,\rmd\varphi$ on the $D^2_i$ (of some
small radius); simply apply Lemma~\ref{lem:exact-area}
to $\Sigma_g^0$ with the discs removed. Moreover, the
Bott integral chosen in Section~\ref{subsubsection:SigmagS1-bott}
may be assumed to be given by $f=r^2$ on the $D^2_i$.

Then one can apply the sewing lemma (Lemma~\ref{lem:sewing})
to the regluing of the solid tori in order to obtain
a Bott-integrable contact structure on the Seifert manifold.
For further details on contact Dehn surgeries of this kind,
see~\cite[Section~4.1]{geig08}.
Given the specific surgery data for a concrete
Seifert manifold, the interpolation of Lutz forms in the
process of Dehn surgery is sufficiently explicit
to allow the computation of homotopical data
of the resulting contact structure.

\begin{rem}
In `most' cases, this construction will result in an
overtwisted contact structure on the Seifert manifold.
Also, the open book decomposition of $\Sigma_g\times S^1$ does not,
in general, induce a natural open book decomposition of the Seifert manifold
obtained by surgery.

However, tight contact structures admitting Bott-integrable Reeb
flows can be constructed on certain Seifert fibred manifolds
with the help of the open book decompositions found
by \"Ozba\u{g}c{\i}~\cite[Proposition~4]{ozba07}.
He describes explicit \emph{horizontal} open books for Seifert manifolds
whose Seifert invariants satisfy a set of inequalities;
`horizontal' means that the open book comes from a surface of section
for the flow defining the Seifert fibration.

The monodromy of these open books is made up of right-handed Dehn twists
along boundary parallel curves. This guarantees that the contact structure
adapted to the open book is tight (even Stein fillable), and our construction
above allows one to show that these structures admit a Bott-integrable
Reeb flow.
\end{rem}
\section{Integrable Reeb flows on the $3$-torus}
\label{section:T3}
In this section we prove Theorem~\ref{thm:S3T3} for the $3$-torus.
\subsection{Contact structures on~$T^3$}
On $T^3=(\R/\Z)^3$ with circular coordinates $x,y,z$,
the $1$-form
\[ \beta_n:=\cos(2\pi nz)\,\rmd x-\sin(2\pi nz)\,\rmd y\]
is a contact form for any $n\in\N$. As shown by
Kanda~\cite{kand97}, Giroux~\cite{giro99}, and Honda~\cite{hond00},
the contact structures
$\eta_n:=\ker\beta_n$ constitute a complete list
of the tight contact structures on $T^3$ up to
diffeomorphism. As tangent $2$-plane fields,
the $\eta_n$ are all homotopic to $\ker\rmd z$ via the
homotopy
\[ (1-t)\beta_n+t\,\rmd z,\;\;\; t\in[0,1],\]
of non-vanishing $1$-forms. In particular, the $\eta_n$
have trivial Euler class.

For the overtwisted contact structures we may appeal
again to Eliashberg's general classification~\cite{elia89}.
Since the homology of $T^3$ is free of $2$-torsion,
the homotopy class of a tangent $2$-plane field over the
$2$-skeleton is determined by its Euler class. The
Euler class of a coorientable tangent $2$-plane bundle on any
closed, orientable $3$-manifold must be even, since its
mod~$2$ reduction is the second Stiefel--Whitney class,
and $3$-manifolds are spin; in fact, parallelisable,
cf.~\cite[Section~4.2]{geig08}.
Any even class in $H^2(T^3;\Z)=\Z^3$ can be realised
as the Euler class of a tangent $2$-plane field.

Our task, then, is to find a Bott integral for the~$\eta_n$,
$n\in\N$, and Bott-integrable overtwisted contact structures
realising every possible Euler class. The $d_3$-invariant,
which determines the homotopy type over the $3$-skeleton,
can be changed at will, thanks to Theorem~\ref{thm:cont-connect},
by forming connected sums with the Bott-integrable
overtwisted contact structures on~$S^3$.

Since we need only realise all contact structures
up to diffeomorphism, the action of the diffeomorphism
group of $T^3$ on homology, which we shall discuss
in Section~\ref{subsection:action}, reduces the task to
realising the Euler classes $(0,0,2m)\in\Z^3=H^2(T^3;\Z)$,
$m\in\N$. (For the trivial Euler class we can take the
connected sum of $(T^3,\eta_n)$ with any of the overtwisted
contact structures on~$S^3$.)
\subsection{A Bott integral for the tight structures}
Since the $\eta_n$ are invariant under the $S^1$-action
generated by the vector field $\partial_z$, we could appeal
to Theorem~\ref{thm:seifert} for the existence of a Bott-integrable
Reeb flow. More simply, we can directly write down a Bott
integral for the Reeb vector field of~$\beta_n$.

Indeed, the Reeb vector field $R_n:=R_{\beta_n}$ is given by
\[ R_n=\cos(2\pi nz)\,\partial_x-\sin(2\pi nz)\,\partial_y.\]
So any Morse function $f_{S^1}$ on the circle gives rise
to a Morse--Bott function
\[ (x,y,z)\longmapsto f_{S^1}(z)\]
invariant under~$R_n$.
\subsection{The action of $\SL(n,\Z)$ on $\Z^n$}
\label{subsection:action}
As preparation for the discussion of the Euler classes realised
by overtwisted contact structures, we analyse the action
of the diffeomorphism group of $T^3$ on (co-)homology.

Every matrix in the special linear group $\SL(n,\Z)$
defines by left-multiplication an action on $\R^n$ that
descends to an action on $T^n$ by orientation-preserving
diffeomorphisms. The induced action on $H_1(T^n;\Z)$
coincides with the action of $\SL(n,\Z)$ on $\Z^n$.
We call an integral vector $(a_1,\ldots,a_n)\in\Z^n$
\emph{primitive} if the $a_1,\ldots, a_n$ are coprime
(not necessarily pairwise), that is, $\gcd(a_1,\ldots,a_n)=1$.
The following proposition is a special case of the lattice
basis extension theorem~\cite[Section~I.2.3]{cass71}, as was
kindly pointed out to us by Frank Vallentin;
we give a simple direct proof.

\begin{prop}
\label{prop:action}
For $n\geq 2$, the group $\SL(n,\Z)$ acts transitively
on primitive elements of~$\Z^n$.
\end{prop}

\begin{proof}
We begin by showing that any primitive vector $(a,b)\in\Z^2$
can be mapped to $(0,1)$ by an element of $\SL(2,\Z)$.
The vector $(a,b)$ being primitive means that $\gcd(a,b)=1$.
By B\'ezout's Lemma there are integers $b,d$ such that
$ad-bc=1$. Then the matrix
\[ \begin{pmatrix}
b & -a\\
d & -c
\end{pmatrix}\]
is in $\SL(2,\Z)$, and it sends $(a,b)^{\ttt}$ to $(0,1)^{\ttt}$.

For $n>2$, we reduce the problem inductively to $n=2$.
Let
\[ (a,b,a_3,\ldots,a_n)\in\Z^n \]
be a primitive vector.
We now find $b,d\in\Z$ such that $ad-bc=\gcd(a,b)$.
Then the matrix
\[ \begin{pmatrix}
b/\gcd(a,b) & -a/\gcd(a,b) & 0 & \cdots & 0\\
d           & -c          & 0 &        &  \\
0           &  0          & 1 &        &  \\
\vdots      &             &   & \ddots &  \\
0           &             &   &        & 1
\end{pmatrix} \] 
is in $\SL(n,\Z)$, and it sends $(a,b,a_3,\ldots,a_n)^{\ttt}$
to the primitive vector
\[ (0,\gcd(a,b),a_3,\ldots,a_n)^{\ttt}. \]
Now iterate this process.
\end{proof}
\subsection{Overtwisted structures from open books}
If we think of $T^3$ as $T^2\times S^1$, the considerations
in Section~\ref{section:open-book} give us a contact
structure $\xi_1$ admitting a Bott-integrable Reeb flow,
which by Lemma~\ref{lem:euler-xig} has Euler class
$e(\xi_1)=-2\PD[S^1]$.

Away from the collars of $\Sigma=\Sigma_1^0$ inside which we
perform the boundary parallel Dehn twists, the mapping torus
is a product, and the Reeb orbits coincide
with the (positively oriented) $S^1$-fibre. The two
binding components $\partial_{\pm}$ are also Reeb orbits. The boundary
orientation of $\partial_{\pm}$ coincides with the orientation of
the Reeb flow, but $[\partial_{\pm}\Sigma]=\pm[S^1]$,
as observed in Section~\ref{subsubsection:euler}.

Arguing as in Section~\ref{subsection:seifert-via-surgery},
we can choose the contact form and the Bott integral such that
by performing Lutz twists along $k$ Reeb orbits representing a
positive $S^1$-fibre, we can obtain Bott-integrable contact structures
$\xi_1^k$ realising any Euler class $e(\xi_1^k)=-2k\PD[S^1]$,
$k\in\N$; see Proposition~4.3.3 and Remark 4.3.4 in \cite{geig08}
for the effect of a single Lutz twist on the Euler class. Notice that
these structures are necessarily overtwisted.

Now, given any even class $e_0\in H^2(T^3;\Z)$, we can write
$e_0$ as a $2k$-fold multiple of a primitive class~$e_0'$.
Thanks to Proposition~\ref{prop:action}, we can find a
diffeomorphism of $T^3$ that pulls back $-\PD[S^1]$ to $e_0'$,
and hence $e(\xi_1^k)$ to~$e_0$. So this Euler class is represented
by a diffeomorphic image of~$\xi_1^k$.

This concludes the proof of Theorem~\ref{thm:S3T3} for the $3$-torus.

\begin{rem}
Alternatively, one may start with one of the tight contact structures
$\eta_n$ on $T^3$ and then introduce isolated elliptic Reeb orbits
in the direction of $\pm\partial_x$ and $\pm\partial_y$ by
the process described in Section~\ref{subsection:create-orbits}.
Lutz twists along such orbits then allow one to
realise any (even) Euler class $(a,b,0)\in \Z^3= H^2(T^3;\Z)$,
where $(1,0,0)=\PD[S^1\times *\times *]$ etc.

Now, given an even class $(a,b,c)\in H^2(T^3;\Z)$, it is straightforward
to write down the required transformation in $\SL(3,\Z)$. With $a',b'\in\Z$
chosen such that $aa'-bb'=\gcd(a,b)$, we have
\[ \begin{pmatrix}
a/\gcd(a,b) & 0 & b'\\
b/\gcd(a,b) & 0 & a'\\
0           & 1 & 0
\end{pmatrix}
\begin{pmatrix}
\gcd(a,b)\\
c\\
0
\end{pmatrix}=
\begin{pmatrix}
a\\
b\\
c
\end{pmatrix}.\]
\end{rem}
\section{Integrable Reeb flows on $S^1\times S^2$}
\label{section:S1S2}
Here we prove Theorem~\ref{thm:S3T3} for $S^1\times S^2$.
This manifold admits a unique tight contact
structure \cite[Theorem~4.10.1]{geig08}, which can be obtained by contact
$(+1)$-surgery on the Legendrian unknot in $(S^3,\xist)$ with
Thurston--Bennequin invariant~$-1$; see~\cite[Lemma~4.3]{dgs04}.
An algorithm developed by Stipsicz~\cite{stip05} translates this
into a supporting open book with page an annulus and trivial monodromy.
As in Section~\ref{section:open-book}, we find a contact form
with a Bott-integrable Reeb flow.

The Euler class of this tight contact structure is trivial.
By a full Lutz twist along an $S^1$-fibre, we obtain a Bott-integrable
overtwisted contact structure with trivial Euler class.
By performing simple Lutz twists along $S^1$-fibres, we can obtain
overtwisted contact structures (admitting a Bott-integrable Reeb flow)
realising the Euler classes $-2k\PD[S^1]$, $k\in\N$. The orientation-preserving
diffeomorphism of $S^1\times S^2$ given by $\theta\mapsto-\theta$
on $S^1$ and the antipodal map on $S^2$ gives us all positive even multiples
of $\PD[S^1]$. Finally, we change the $d_3$-invariant at will by
connected sums with the overtwisted contact structures on~$S^3$.
\section{Integrable Reeb flows containing singular Klein bottles}
\label{section:klein}
Prime $3$-manifolds admitting embedded Klein bottles are
relatively rare; see~\cite{geth}. For instance, one cannot
embed a Klein bottle into $S^3$, and the only lens spaces
containing embedded Klein bottles are the $L(4n,2n\pm 1)$.

In this section we show how to obtain integrable
Reeb flows whose Bott integral contains a Klein
bottle in its critical set. Also, by way of example,
we describe how to perturb the Bott integral so as
the make the Klein bottle disappear from the critical set.
\subsection{Constructing Klein bottles in the critical set}
As discussed in~\cite{geth}, the (closed) tubular neighbourhood
$\nu\calK$ of a Klein bottle $\calK$ embedded in an orientable
$3$-manifold can be described by
\[ \nu\calK=\bigl([0,1]\times [-1,1]\times S^1\bigr)/
(1,r,\theta)\sim(0,-r,-\theta),\]
with $\calK\subset\nu\calK$ given by $\{r=0\}$.
Notice that $\partial(\nu\calK)$ is a $2$-torus.
We write $t$ for the coordinate in the first factor,
the interval $[0,1]$. The flow of $\partial_t$
defines a Seifert fibration of $\nu\calK$
with two singular fibres of multiplicity~$2$,
and the quotient is $D^2(2,2)$, a disc with
two orbifold points of order~$2$.

As collar coordinates of the annulus $[-1,1]\times S^1$ near
$r=1$ we can take $(\rho_+,\theta_+)=(r-1,\theta)$;
near $r=-1$ we choose $(\rho_-,\theta_-)=(-r-1,-\theta)$.
This choice on either collar is consistent with the orientation
of the annulus defined by $\rmd r\wedge\rmd\theta$.
The map $(r,\theta)\mapsto (-r,-\theta)$ interchanges
$(\rho_+,\theta_+)$ and $(\rho_-,\theta_-)$. This gives
us well-defined collar coordinates $(t\,\mathrm{mod}\,1,\rho,\theta)$
on the quotient $\nu\calK$.

On $[0,1]\times [-1,1]\times S^1$ we can define the
contact form $\alpha:=\rmd t+r\,\rmd\theta$ with Reeb vector
field $R_{\alpha}=\partial_t$. This descends to the quotient
$\nu\calK$, and in the collar coordinates this
contact form is given by $\rmd t+(1+\rho)\,\rmd\theta$.
This is a Lutz form, with $t\,\mathrm{mod}\,1$ and $\theta$
the torus coordinates, and $\rho$ the transverse coordinate.

The function $f([t,r,\theta])=r^2$ is well defined on~$\nu\calK$,
and on the collar of $\nu\calK$ this function equals
$(1+\rho)^2$. This function is a Bott integral for~$R_{\alpha}$,
with $\calK=\{r=0\}$ as critical set.

As in Section~\ref{subsection:seifert-via-surgery},
where we described the extension of Bott-integrable
contact structures over solid tori glued in during the
process of Dehn surgery, one sees in the present situation
that there is a Bott-integrable extension of $\alpha$
to any Dehn filling of $\nu\calK$, i.e.\ any closed
$3$-manifold obtained by gluing a solid torus to
$\partial(\nu\calK)$. As discussed in \cite{geth},
amongst lens spaces precisely the
$L(4n,2n\pm 1)$ can be realised in this fashion.

An inspection of the proof of \cite[Lemma~4.4]{geth}
shows that in the case of the lens spaces $L(4n,2n\pm 1)$, the
extension of the contact form can be chosen in such a way
that the Reeb flow defines the Seifert fibration of
$L(4n,2n\pm 1)$ over $S^2(2,2)$, the $2$-sphere with
two orbifold points of order~$2$. The critical set
of the Bott integral can be arranged to consist of
$\calK$ and a single isolated periodic Reeb orbit,
namely, the spine of the solid torus making
up the Dehn filling.

A global description of this Reeb flow can be given
with the help of~\cite{gela21}. As shown there,
$L(4n,2n-1)$ can be described as the quotient of
$S^3\subset\C^2$ under the $\Z_{4n}$-action generated by
\[ (z_1,z_2)\longmapsto(\rme^{\pi\rmi/2n}\oz_2,
\rme^{-\pi\rmi/2n}\oz_1);\]
for $L(4n,2n+1)$, which is obtained by reversing
the global orientation, there is a similar description.
This action is equivariant with respect to the
anti-Hopf flow
\[ (z_1,z_2)\longmapsto(\rme^{\rmi\theta}z_1,
\rme^{-\rmi\theta}z_2).\]
The connection $1$-form
\[ \frac{\rmi}{2}(\rmd z_1\wedge\rmd\oz_1-
\rmd z_2\wedge\rmd\oz_2)\]
of the anti-Hopf fibration is a contact form that
descends to $S^3/\Z_{4n}$. The anti-Hopf flow
descends to the Reeb flow
of this contact form on $S^3/\Z_{4n}$, and this
is precisely the Reeb flow obtained from the
Dehn filling of~$\nu\calK$.

The $\Z_{4n}$-action also preserves the fibres of
the Hopf fibration, but the generator reverses the fibre
orientation. The Hopf fibration descends to a non-orientable
Seifert fibration of $L(4n,2n-1)$ over $\RP^2(n)$.
The Bott integral can be defined as the lift of
a radially symmetric function on $\RP^2(n)$ with an isolated
critical point in the orbifold point taken as the centre, and a
critical circle being the $\RP^1$ at infinity.
\subsection{Removing Klein bottles from the critical set}
We now show how to perturb the Bott integral so as
to remove the critical Klein bottle and introduce two
isolated critical Reeb orbits instead. This illustrates
the Reeb analogue of the genericity and perturbation
results of Kalashnikov~\cite{kala95} concerning
$4$-dimensional Hamiltonian systems.

We start with the function $f([t,r,\theta])=r^2$
on~$\nu\calK$ with a critical Klein bottle
$\calK=\{r=0\}$. For some small
$\varepsilon >0$, let $\chi\co[-1,1]\rightarrow
[0,\varepsilon^2]$ be a smooth function with
$\chi\equiv\varepsilon^2$ on the interval
$[-\varepsilon,\varepsilon]$, and $\chi\equiv 0$
on the intervals $[-1,-2\varepsilon]$ and
$[2\varepsilon,1]$. On the intervals
$[-2\varepsilon,\-\varepsilon]$ and
$[\varepsilon,2\varepsilon]$ we may assume
that $|\chi'(r)|<|2r|$.

Now replace the old $f$ by
\[ f([t,r,\theta]):=r^2+\chi(r)\cos\theta.\]
This function is still invariant under the flow
of $R_{\alpha}=\partial_t$, and the differential
\[ \rmd f=(2r+\chi'(r)\cos\theta)\,\rmd r-\chi(r)\sin\theta\,
\rmd\theta\]
vanishes only when $r=0$ and $\theta\in\{0,\pi\}$.
This describes precisely the two singular fibres
of the Seifert fibration $\nu\calK\rightarrow D^2(2,2)$.
The Hessian of $f$ at those critical points is
\[ \begin{pmatrix}
2 & 0\\
0 & -\varepsilon^2\cos\theta
\end{pmatrix}
=\begin{pmatrix}
2 & 0\\
0 & \mp\varepsilon^2
\end{pmatrix},\]
so we have created an elliptic Reeb orbit (along which $f$ is minimal)
and a hyperbolic one.
\begin{ack}
This paper was fostered by conversations with Marcelo
Alves and Otto van Koert, some of them at the Lorentz Center, Leiden,
during the 2022 workshop on Symplectic Dynamics Beyond Periodic Orbits. 
We thank Eva Miranda, Peter Albers, Simon Vialaret and Kai Zehmisch for
listening to and encouraging our developing ideas.

This work is part of a project in the
Son\-der\-for\-schungs\-be\-reich TRR 191
\textit{Symplectic Structures in Geometry, Algebra and Dynamics},
funded by the DFG (Projektnummer 281071066 -- TRR 191).
\end{ack}

\end{document}